\documentclass{elsarticle}

\usepackage{xcolor}

\usepackage{graphics,graphicx}
\usepackage{tikz-cd}
\usetikzlibrary{calc}
\usetikzlibrary{arrows} 

\usepackage[utf8]{inputenc}
\usepackage[T1]{fontenc}

\usepackage{amsmath}
\usepackage{amsthm}
\usepackage{amssymb}

\usepackage{enumitem}

\usepackage{xfrac}

\usepackage{booktabs}

\usepackage{constants}

\usepackage{tensor}

\newtheorem{thm}{Theorem}[section]
\newtheorem{cor}[thm]{Corollary}
\newtheorem{lem}[thm]{Lemma}
\newtheorem{prop}[thm]{Proposition}

\theoremstyle{definition}

\newtheorem{assumption}[thm]{Assumption}

\DeclareMathOperator{\Ric}{Ric}
\DeclareMathOperator{\SL}{SL}

\DeclareMathOperator{\SO}{SO}
\DeclareMathOperator{\GL}{GL}

\DeclareMathOperator{\grad}{grad}

\DeclareMathOperator{\Vol}{Vol}

\newcommand{\bbC}{\mathbb{C}}
\newcommand{\bbZ}{\mathbb{Z}}
\newcommand{\bbR}{\mathbb{R}}
\newcommand{\bbQ}{\mathbb{Q}}
\newcommand{\bbP}{\mathbb{P}}
\newcommand{\rf}{\mathrm{ref}}
\newcommand{\tor}{\mathrm{tor}}

\newcommand{\MAR}[1]{\det(d^2#1)} 
\newcommand{\lalg}[1]{\mathfrak{#1}}
\newcommand{\Potent}[1]{\mathcal{P}(\Delta_{#1})}
\newcommand{\fun}[2]{{u}^{#2}_{#1}}
\newcommand{\barDH}[1]{\mathbf{bar}_{#1}}
\newcommand{\iddbar}{\sqrt{-1}\partial\bar\partial}

\begin{document}

\title{Coupled complex Monge-Ampère equations on Fano horosymmetric manifolds}

\author[1]{Thibaut Delcroix\corref{cor1}} 
\ead{thibaut.delcroix@umontpellier.fr}
\address[1]{Thibaut Delcroix \\
Univ Montpellier, CNRS \\ Montpellier, France
}

\author[2]{Jakob Hultgren}
\ead{hultgren@umd.edu}
\address[2]{Jakob Hultgren \\
University of Maryland\\
Department of Mathematics\\
4176 Campus Drive \\
College Park, MD 20742-4015, USA
}

\cortext[cor1]{Corresponding author}

\date{}


\begin{abstract}
We address a general system of complex Monge-Ampère equations on Fano horosymmetric manifolds and give necessary and sufficient conditions for existence of solutions in terms of combinatorial data of the manifold. This gives new results about Mabuchi metrics, twisted Kähler-Einstein metrics and coupled Kähler-Ricci solitons and provides a unified approach to many previous results on canonical metrics on Kähler manifolds. 
\end{abstract}

\maketitle

\section*{Résumé}

Nous considérons un système d'équation de Monge-Ampère complexes général sur les variétés horosymétriques Fano, et nous obtenons des conditions nécessaires et suffisantes d'existence de solutions en termes des données combinatoires associées à de telles variétés. Nous appliquons ce résultat général pour obtenir de nouveau résultats d'existence ou de non-existence de métriques de Mabuchi, de métriques de Kähler-Einstein tordues, de solitons de Kähler-Ricci couplés, et pour fournir une approche unifiée à de nombreux résultats de la littérature sur les métriques canoniques sur les variétés Kähler.

\section{Introduction}
The background for the present paper is given by several recent results in Kähler geometry which uses combinatorial data to give necessary and sufficient conditions for existence of canonical metrics on various types of Kähler manifolds. The starting point for this branch of Kähler geometry is the celebrated result by Wang and Zhu \cite{WZ04} which completely characterize toric Kähler-Einstein manifolds among toric Fano manifolds in terms of their moment polytope. This has led to many subsequent developments. To name a few: \cite{BB13} regarding Kähler-Ricci solitons on log Fano varieties, \cite{DelKE} regarding Kähler-Einstein metrics on group compactification, \cite{Yao} regarding Mabuchi metrics on toric Fano varieties and \cite{Hul} regarding coupled Kähler-Einstein metrics on toric Fano manifolds. The purpose of this paper is to add several new results and give a unified approach to the existing results above. This is done by providing $C^0$-estimates to a class of coupled real Monge-Ampère equations on polyhedral cones (see Theorem~\ref{thm:C0Estimates} below). 
As applications we get the following new results:
\begin{itemize}
\item A necessary and sufficient condition for existence of Mabuchi metrics on horosymmetric Fano manifolds,
\item A necessary and sufficient condition for existence of coupled Kähler-Ricci solitons on horosymmetric Fano manifolds,
\item A necessary and sufficient condition for existence of twisted Kähler-Einstein metrics on toric Fano manifolds (with a possibly singular, semi-positive, torus invariant twist)
\end{itemize}
Along more concrete lines, the results above are used in a subsequent paper \cite{DelExa} by the first author to exhibit the following examples:
 \begin{itemize}
\item An infinite family of Fano manifolds which admit coupled Kähler-Einstein metrics but no Kähler-Einstein metrics,
\item A family of manifolds which admit both Kähler-Ricci solitons and multiplier Hermitian structures of ‘log-polynomial’ type,
\item New examples of Fano manifold which don’t admit Mabuchi metrics.
\end{itemize}

\subsection{Detailed Background}
For the last 40 years, the study of Kähler-Einstein metrics, i.e. Kähler metrics $\omega$ that satisfy the Einstein condition 
$$ \Ric \omega = \lambda \omega $$
where $\Ric \omega$ is the Ricci curvature form of $\omega$ and $\lambda\in \bbR$ has been an important part of Kähler geometry. 
Existence of Kähler-Einstein metrics 
has turned out to be connected to subtle properties in algebraic geometry. By the Yau-Tian-Donaldson conjecture which was recently confirmed in \cite{CDS15a,CDS15b,CDS15c} and \cite{Tia15}, a compact Kähler manifold admits a Kähler-Einstein metric if and only if it is K-stable, an algebraic condition due to Tian and Donaldson which is modeled on Geometric Invariant Theory. 
This result has since been extended to twisted Kähler-Einstein metrics and Kähler-Ricci solitons \cite{DS16} and, very recently, to Mabuchi metrics \cite{HL20}. 
It should be said, however, that even in the light of \cite{CDS15a,CDS15b,CDS15c}, \cite{DS16} and \cite{HL20}, determining if a given manifold admits a (twisted) Kähler-Einstein metric, Kähler-Ricci soliton or a Mabuchi metrics is not a straight forward task. The condition of K-stability is not readily checkable. Consequently, along a line parallel to the Yau-Tian-Donaldson conjecture, much work has been done to find simple, explicit conditions for existence of Kähler-Einstein metrics on special classes of manifolds. A starting point this is the result of Wang and Zhu in 2004 \cite{WZ04}, saying that a toric Fano manifold admits a Kähler-Einstein metric if and only if the barycenter of the associated polytope vanishes. In addition, Wang and Zhu proved that any toric Fano manifold admit a Kähler-Ricci soliton. 

In this paper we study a general system of Monge-Ampère equations which incorporates the three generalizations of Kähler-Einstein metrics mentioned above as well as coupled Kähler-Einstein metrics (see Equation~\ref{eqn_general} below). Our main theorem (Theorem~\ref{thm:cpld_can_horo} below) is a necessary and sufficient condition for existence of solutions to this equation on horosymmetric manifolds, a class of manifolds introduced by the first named author in \cite{DelHoro}. This class is strictly included in the class of spherical manifolds and strictly includes the class of group compactifications. We deduce this theorem as a consequence of two independent results. 
The first of these is that Yau's higher order estimates along the continuity method for Monge-Ampère equations, as well as local solvability, apply to a suitable continuity path containing Equation~\ref{eqn_general} (see Theorem~\ref{thm:ReductionToC0} below). The second of these is that the $C^0$-estimates on toric manifolds proved by Wang and Zhu extends to the current setting of Equation~\ref{eqn_general} on horosymmetric manifolds. This is done by proving $C^0$-estimates along a continuity path of real Monge-Ampère equations on convex polyhedral cones (see Theorem~\ref{thm:C0Estimates} below). This generalizes and provides a simpler proof of a related $C^0$-estimate on group compactifications in \cite{DelKE}. As a byproduct we get necessary conditions for existence of solutions to these system of real Monge-Ampère equations. Moreover, these conditions are also sufficient in the cases when the data defining the cones is induced by horosymmetric manifolds. 

The paper is organized as follows: in the remaining of this introductory section, we state our main results. Section~\ref{sec:reductionC0} is devoted to the proof of Theorem~\ref{thm:ReductionToC0}. The proof of Theorem~\ref{thm:C0Estimates} is given in  Section~\ref{sec_C0}. We then explain in Section~\ref{sec:horosym} how these two results apply to the case of horosymmetric manifolds, after recalling the definitions and tools to deal with such manifolds. Finally, we include in Section~\ref{sec:examples} an illustration of our results on low dimensional horosymmetric manifolds.  

\subsection{Setup}\label{sec:setup}
We will now formulate the system of Monge-Ampère equations we will consider. Let $X$ be a Fano manifold of dimension $n$ and $K$ a compact subgroup of the automorphism group $\mathrm{Aut}(X)$ of $X$. 
For a positive integer $k$, fix $K$-invariant Kähler forms $\theta_1,\ldots,\theta_k$ and a $K$-invariant semi-positive (1,1)-form $\gamma$ such that
$$ [\gamma] + \sum_{i=1}^{k}[\theta_i]=c_1(X).$$ 
Let $\theta_0$ denote a Kähler form of unit mass such that 
$$\gamma + \sum_{i=1}^{k}\theta_i = \Ric(\theta_0).$$ 
Fix a $k$-tuple of (real) holomorphic vector fields $(V_i)$, commuting with the action of $K$ 
such that the subgroup generated by $JV_i$ lies in $K$ for each $i$ (where $J$ denotes the complex structure on $X$). 
Let $f_i$ denote the $K$-invariant (real valued) $\iddbar$-potentials for the Lie derivatives $L_{V_i}\theta_i$, 
normalized so that $\int_Xf_i\theta_i^n=0$. 
Let $(h_i)$ denote a $k$-tuple of real valued smooth concave functions 
on the real line which satisfy $\int_Xe^{h_i\circ f_i}\theta_i^n=1$.

We consider in this paper systems of equations whose solutions are $k$-tuples 
$(\phi_i) = (\phi_1,\ldots,\phi_k)$ of $K$-invariant real-valued functions on $X$ such that for all $i\in \{1,\ldots,k\}$,
$\theta_i+\iddbar \phi_i$ is Kähler and 
\begin{equation} 
\label{eqn_general}
e^{h_i(f_i+V_i(\phi_i))}(\theta_i+\iddbar \phi_i)^n=
e^{-\sum_{m=1}^k\phi_m}\theta_0^n. 
\end{equation}

When $\gamma=0$ and $k=1$, this equation reduces to an equation considered by Mabuchi in \cite{Mab03} and whose solutions define what he calls multiplier hermitian structures. Different choices of $h_1$ recover the definitions of Kähler-Einstein metrics (if $h_1$ is constant), Kähler-Ricci solitons (if $h_1$ is affine) 
and Mabuchi metrics (if $e^{h_1}$ is affine).

When $\gamma=0$, $k>1$ and $h_i$ is affine we get the system of equations defining coupled Kähler-Einstein metrics and coupled Kähler-Ricci solitons introduced in \cite{HWN17} and \cite{Hul}. Moreover, putting $\gamma=0$, $k>1$ and choosing $e^{h_i}$ affine we get a natural definition for \emph{coupled Mabuchi metrics}, generalizing both Mabuchi metrics and coupled Kähler-Einstein metrics. 

Finally, choosing a non-zero $\gamma$ allows to consider twisted versions of the above equations. 

In terms of the associated Kähler forms $\omega_i=\theta_i+\iddbar\phi_i$, \eqref{eqn_general} is equivalent to 
$$
\Ric \omega_i - \sqrt{-1} \partial \bar \partial h_i(f_i+V_i(\phi_i)) = \gamma + \sum_{m=1}^k \omega_m
$$
for $i\in \{1,\ldots,k\}$. If the functions $h_1,\ldots,h_k$ are invertible, and if $\gamma=0$ to simplify, a third formulation is given by considering the coupled Ricci potentials, i.e. functions $F_1,\ldots,F_k$ such that
$$ \iddbar F_i = \Ric \omega_i - \sum_{m=1}^k \omega_m. $$ 
Then $(\phi_i)$ solves \eqref{eqn_general} if and only if $h_i^{-1}\circ F_i$ is a $\iddbar$-potential of $L_{V_i}(\omega_i)$ for each $i$.

Note that in the case of Mabuchi metrics, it is implicitly part of our assumptions that the affine function $e^{h_1\circ f_1}$ is positive. This condition was identified by Mabuchi as a necessary condition for existence of Mabuchi metrics and it is actually of the same nature as the assumption that $X$ is Fano that we make to consider positive Kähler-Einstein metrics or Kähler-Ricci solitons.

\subsection{A continuity path and reduction to $C^0$-estimates}
\label{sec:intro_red}
We will consider the following continuity path for \eqref{eqn_general}: 
\begin{equation}
e^{h_i(f_i+V_i(\phi_i))}(\theta_i+\iddbar \phi_i)^n=
e^{-t\sum_{m=1}^k\phi_m}\theta_0^n. 
\label{eq:ComplexContPath}
\end{equation}
where $0\leq t \leq 1$. 
We must emphasize that Equation~\eqref{eq:ComplexContPath} is really a family of equations indexed by $t$, and as such solutions are families of tuples of functions $(\phi_1^t,\ldots,\phi_k^t)$ indexed by $t$. However, to simplify notations, we will in the following denote a tuple of solutions at a given $t$ by $(\phi_1,\ldots,\phi_k)$, omitting the index $t$. 
Note that given a solution $(\phi_1,\ldots,\phi_k)$ and constants $C_1,\ldots, C_k$ such that $\sum_{m=1}^k C_m=0$, we get a new solution by considering the tuple $(\phi_i+C_i)_{i=1}^k$.  Fixing a point $x_0\in X$ we will assume any solution $(\phi_1,\ldots,\phi_k)$ of \eqref{eq:ComplexContPath} to be normalized to satisfy 
\begin{equation}
    \label{eq:Normalization}
    \phi_1(x_0) = \ldots = \phi_k(x_0).
\end{equation}

Given a continuity path of the type \eqref{eq:ComplexContPath}, we will say that $C_0$-estimates hold on an interval $[t_0,t]\subset [0,1]$ if there exists a constant such that for any solution $(\phi_i)$ at $t'\in [t_0,t]$
\begin{equation}
        \label{eq:C0Assumption}
        \max_i\sup_X |\phi_i| < C.
\end{equation}
Using well-known techniques of Yau and Aubin we will reduce existence of solutions of \eqref{eqn_general} to a priori $C^0$-estimates for \eqref{eq:ComplexContPath}. 
\begin{thm}
    \label{thm:ReductionToC0}
    Let $t\in (0,1]$ and assume that for each $t_0\in (0,t)$, $C^0$-estimates hold for \eqref{eq:ComplexContPath} on $[t_0,t]$. Then \eqref{eq:ComplexContPath} has a solution for any $t'\in [0,t]$. In particular, if $t=1$ then \eqref{eqn_general} has a solution. 
\end{thm}

\subsection{Canonical metrics of Kähler-Einstein type on horosymmetric manifolds} 
\label{sec:intro_horo}

On a manifold $X$ equipped with an automorphism group of large dimension, it often turned out that complex Monge-Ampère equations such as the Kähler-Einstein equation could be translated into a real Monge-Ampère equation with data encoded by combinatorial information on the action of $\mathrm{Aut}(X)$. 
The major example is that of toric manifolds \cite{WZ04} but several bigger such classes of manifolds were studied over the years, notably homogeneous toric bundles \cite{PS10} and group compactifications \cite{DelKE}. 
Horosymmetric manifolds were introduced by the first author in \cite{DelHoro} as a generalization of the above classes, and the tools to translate complex Monge-Ampère equations into real Monge-Ampère equations were developed in the same paper. 

In the present article, we will study our general system of complex Monge-Ampère equations \eqref{eqn_general} on horosymmetric manifolds. 
It yields a complete combinatorial characterization of existence of solutions (under a few additional mild assumptions such as invariance of the data under a compact subgroup).
In particular we recover and generalize a wide array of recent results on existence of canonical metrics. 

A horosymmetric manifold is a manifold $X$ equipped with an action of a connected linear reductive complex group $G$ such that the action admits an open orbit $G\cdot x$ which is a homogeneous bundle over a generalized flag manifold with fiber a complex symmetric space. 
Several combinatorial data were associated to such manifolds and their line bundles in \cite{DelHoro}. 
For this introduction we just need the following. 
We choose (wisely) a maximal torus $T$ and Borel subgroup $T\subset B$ of $G$. 
The lattice $\mathcal{M}=\mathfrak{X}(T/T_x) \subset \mathfrak{X}(T)$ is called the spherical lattice.
The root system of $G$ splits into three (possibly empty) parts 
$\Phi^+=\Phi_{Q^u}\cup \Phi_s^+\cup (\Phi_L^+)^{\sigma}$
(see Section~\ref{sec:horosym} for details).
Define $\bar{C}^+$ as the cone of elements $p\in\mathcal{M}\otimes \bbR$ such that $\kappa(p,\alpha)\geq 0$ for all $\alpha\in\Phi_s^+$, where $\kappa$ denotes the Killing form. 

To a nef class $[\gamma]$ on the horosymmetric manifold $X$ we associate, by elaborating on \cite{DelHoro}, a \emph{toric polytope} $\Delta^{\tor}_{\gamma}\subset \mathcal{M}\otimes \bbR$ (well defined up to translation by an element of $\mathfrak{X}(T/[G,G])\otimes \bbR$) and an isotropy character $\chi_i\in \mathfrak{X}(T)$. 
Furthermore, for the anticanonical line bundle, there is a canonical choice of representative $\Delta^{\tor}_{ac}$ among the possible toric polytopes. 

We now consider the setup of Section~\ref{sec:setup} on a horosymmetric Fano manifold, with the additional assumption that the $\theta_i$ and $\gamma$ are invariant under a fixed maximal compact subgroup $K$ of $G$, that the vector fields $V_i$ commute with the action of $G$, and that the classes $[\theta_i]$ and $[\gamma]$ are in the subspace generated by semiample line bundles whose restriction to the fiber of the open orbit are trivial.
Note that horosymmetric manifolds are Mori Dream Spaces and as such their nef cone is generated by classes of line bundles. The subspace defined above always contain the anticanonical class and is equal to the full nef cone as long as the symmetric fiber does not have any Hermitian factor. 
The assumption on the vector fields allows to associate to $V_i$ an affine function $\ell_i$ on $\mathcal{M}\otimes \bbR$ (see Section~\ref{sec:HamiltHoro}). 

Using the assumption $[\gamma]+\sum_{i=1}^k[\theta_i]=c_1(X)$ we may and do choose the toric polytopes of each class so that the Minkowski sum $\Delta^{\tor}_{\gamma}+\sum_{i=1}^k\Delta_{\theta_i}^{\tor}$ is equal to the canonical $\Delta^{\tor}_{ac}$.
For the statement of the main theorem, we introduce the following notations. 
The Duistermaat-Heckman polynomial $P_{DH}$ is defined on $\mathfrak{X}(T)\otimes\bbR$ by 
\[ P_{DH}(q):=\prod_{\alpha\in\Phi_{Q^u}\cup \Phi_s^+}\kappa(\alpha,q). \]
Let $2\rho_H=\sum_{\alpha\in\Phi_{Q^u}\cup \Phi_s^+}\alpha - \chi_{ac}$ where $\chi_{ac}$ is the isotropy character of the anticanonical line bundle. Define the $\ell_i$-modified Duistermaat-Heckman barycenters by 
\begin{equation} 
\label{eq:defn_bar}
\mathbf{bar}^{DH}_i= \int_{\Delta^{\tor}_{\theta_i}\cap \bar{C}^+} p e^{h_i\circ \ell_i (p)}P_{DH}(\chi_i+p)dp 
\end{equation}
where $dp$ is the Lebesgue measure normalized so that $\int_{\Delta^{\tor}_{\theta_i}\cap \bar{C}^+} e^{h_i\circ \ell_i (p)}P_{DH}(\chi_i+p)dp=1$.

We are now ready to state our main result.
\begin{thm}
\label{thm:cpld_can_horo}
On a Fano horosymmetric manifold $X$, there exists a solution to the system \eqref{eq:ComplexContPath} if and only if 
\[ 0 \in \mathrm{Relint}\left( t\sum_{i=1}^k \mathbf{bar}^{DH}_i+(1-t)\sum_{i=1}^k \Delta^{\tor}_{\theta_i} +\Delta^{\tor}_{\gamma}-2\rho_H+-(\bar{C}^+)^{\vee} \right) \]
\end{thm}

As a consequence, we obtain numerous generalizations or alternate proofs of recent results, and we shall illustrate this with a few corollaries. 
Let us first consider the non-coupled and non-twisted case (so $\Delta_{\theta_1}^{\tor}=\Delta_{ac}^{\tor}$), and first the case when $h_1$ is constant. 
We obtain the following generalization of \cite{WZ04,PS10,Li11,DelKE,Yao17} 
as well as an alternate proof to a particular case of \cite{DelKSSV}.
\begin{cor}
A horosymmetric Fano manifold $X$ is Kähler-Einstein if and only if 
$\mathbf{bar}^{DH}_1-2\rho_H\in \mathrm{Relint}((\bar{C}^+)^{\vee})$.
The greatest Ricci lower bound $R(X)$ of a horosymmetric Fano manifold $X$ is equal to 
\[ \sup\left\{t\in ]0,1[~;~ 2\rho_H+\frac{t}{1-t}(2\rho_H-\mathbf{bar}^{DH}_1) \in \mathrm{Relint}(\Delta_{ac}^{\tor}-(\bar{C}^+)^{\vee}) \right\} \]
\end{cor}
The variant for Kähler-Ricci solitons and greatest Bakry-Emery-Ricci lower bounds follow from allowing $h_1$ to be an affine function. Note that the possible affine function is fully determined up to constant by the conditions ensuring that the translated polytope lies in the linear span of $(\bar{C}^+)^{\vee}$. 

The case of Mabuchi metrics (see \cite{Mab01}) is solved by allowing $h_1$ to be the logarithm of an affine function, and we obtain a generalization and alternate proof of \cite{Yao, LZ17}. 
One can actually formulate the problem purely in terms of the polytope:

\begin{cor}
The horosymmetric Fano manifold $X$ admits a Mabuchi metric if and only there exists $\xi\in (\mathfrak{X}(T/[G,G])\otimes \bbR)^*$ such that replacing $e^{h_1\circ\ell_1}$ by $\xi+C$ in the definition of the barycenter (for an appropriate normalizing constant $C$), 
both of the following conditions are satisfied:
\begin{itemize}
    \item $\xi(p)+C>0$ on $\Delta_{ac}^{\tor}\cap \bar{C}^+$, and 
    \item $\mathbf{bar}^{DH}_1-2\rho_H\in \mathrm{Relint}((\bar{C}^+)^{\vee})$.
\end{itemize}
\end{cor}

Finally, consider the case of coupled Kähler-Ricci solitons (see \cite{HWN17}, \cite{Pin18} and \cite{FZ}) on horosymmetric manifolds. We obtain the following generalization of the second author's existence result for toric manifolds \cite{Hul}. We place ourselves in the general coupled but not twisted setting, and assume that all the $h_i$ are affine. Again we may forget the $h_i$ and $\ell_i$ to formulate the problem only in terms of the polytope and the data of elements $\xi_i\in (\mathfrak{X}(T/[G,G])\otimes \bbR)^*$ and normalizing constants $C_i$.
\begin{cor}
The decomposition $\sum_{i=1}^k \theta_i=c_1(X)$ admits coupled Kähler-Ricci solitons if and only if there exists $k$ elements $\xi_1,\ldots, \xi_k \in (\mathfrak{X}(T/[G,G])\otimes \bbR)^*$ such that 
$\sum_{i=1}^k \mathbf{bar}_i^{DH}-2\rho_H\in \mathrm{Relint}((\bar{C}^+)^{\vee})$ where $h_i\circ\ell_i$ is replaced by $\xi_i+C_i$ in the definition of the , with appropriate normalizing constants $C_i$.
\end{cor}

Finally we note that there are new results in the toric case. 
The obvious new results are for the coupled canonical metrics we have defined in the present paper of course, but it seems also noticeable that our theorem, in the simpler case of twisted Kähler-Einstein metrics, was never observed and proved before. 
\begin{cor}
Let $X$ be toric Fano manifold. Let $\gamma$ be a $K$-invariant semipositive $(1,1)$-form and $\theta$ a Kähler form such that $\theta+\gamma\in c_1(X)$. 
Let $\Delta_{\gamma}$ be the polytope associated to $\gamma$ and let $\mathbf{bar}$ be the barycenter of the moment polytope of $[\theta]$. Then there exists a solution $\omega\in [\theta]$ to the twisted Kähler-Einstein equation $\Ric(\omega)=\omega+\gamma$ if and only if $0\in \mathrm{Relint}(\mathbf{bar}+\Delta_{\gamma})$.
\end{cor}

In the last section of the article we will illustrate these corollaries on some low dimensional examples of horosymmetric manifolds. 
We take advantage of this occasion to determine the best horosymmetric structure on Fano threefolds, that is, given a horosymmetric Fano threefold $X$, we determine the largest connected reductive subgroup $G$ of its automorphism group such that $X$ is horosymmetric under the action of $G$. Most are toric, but admit a lower rank horospherical structure, and some are symmetric but not horospherical. We obtain that way new examples of Fano threefolds with Mabuchi metrics or without Mabuchi metrics. 

We investigate on one rank one horospherical example the existence of a different type of canonical metric: instead of taking $h_1$ to be the logarithm of an affine function as in the case of Mabuchi metrics, we allow an integral multiple of such. In other words, we consider the case when $e^{h_1}$ is a power of an affine function. This provides an interpolation between the case of Mabuchi metrics and the case of Kähler-Ricci solitons. We illustrate on the example the fact that there may be no Mabuchi metrics, but a canonical metric for $e^{h_1}$ the square of an affine function. 

Finally, we consider two rank one horospherical Fano fourfolds which admit coupled Kähler-Einstein pairs but no single Kähler-Einstein metric. For one, it is just the horospherical point of view on the second author's example \cite{Hul}, and the other is a natural variant. 

\subsection{Real Monge-Ampère equations on polyhedral cones}
As mentioned in the previous section, on horosymmetric manifolds \eqref{eq:ComplexContPath} reduces to a real Monge-Ampère equation. One main ingredient in the proof of Theorem~\ref{thm:cpld_can_horo} is an a priori $C^0$-estimate for these real Monge-Ampère equations. The convex geometric setting we will use to state these $C^0$-estimates is a slight generalization of the one proposed by horosymmetric manifolds. 

Let $\lalg{a}$ be a (real) vector space, $r=\dim \lalg{a}$ and  $\lalg{a}^+$ be a convex polyhedral cone in $\lalg{a}$. We will implicitly fix a basis of $\lalg{a}$ and corresponding norms $|\cdot|$ and Lebesgue measures $dx$ on $\lalg{a}$ and $dp$ on its dual $\mathfrak{a}^*$. 
Moreover, let $J$ be a continuous function on $\lalg{a}^+$, positive on $\mathring{\lalg{a}}^+$ and vanishing on $\partial \lalg{a}$. Assume also the function defined on $\mathring{\lalg{a}}^+$ by $j=-\log J$ is smooth and convex. Finally, let $\Delta_1,\ldots,\Delta_k$ be convex bodies in $\mathfrak{a}^*$ and, for each $i\in \{1,\ldots,k\}$, let $G_i$ be a continuous function on $\Delta_i$, smooth and positive on  $\mathring{\Delta}_i$, and such that 
$ \int_{\Delta_i} G_i dp = 1$, where the integral is taken with respect to the fixed Lebesgue measure $dp$ on $\lalg{a}^*$.

The statement of the theorem will involve a few technical conditions on $j$ and $G_1,\ldots,G_k$. To state them we first need some terminology. Given a convex function $f: \mathring{\lalg{a}}^+\rightarrow \bbR$,  we define its 
\emph{asymptotic function} 
\[ f_{\infty}:\lalg{a}^+\rightarrow \bbR\cup \{\infty\} 
\qquad \xi \mapsto \lim_{t\rightarrow \infty} f(x+t\xi)/t \]
for any choice of $x\in \mathring{\lalg{a}}^+$ (it is standard that for a convex function, 
$f_{\infty}$ does not depend on $x$). Moreover, for any convex body $\Delta\subset \mathfrak{a}^*$, let $\mathcal{P}(\Delta)$ denote the space of all smooth, strictly convex functions $u:\mathfrak{a}\rightarrow \bbR$ such that $\overline{du(\lalg{a}^+)} = \Delta$ and
\begin{equation} 
\label{eq:PDelta}
\sup_{\lalg{a}^+} |u-v_{\Delta}| < \infty, 
\end{equation} 
where $v_\Delta$ is the support function of $\Delta$. We note that $f_{\infty}=v_{\Delta}|_{\lalg{a}^+}$ for any 
$f\in \mathcal{P}(\Delta)$. 

From now on, we will let $\Delta$ denote the Minkowski sum $\Delta = \sum_{i=1}^k \Delta_i$. 

In order to prove our main result on coupled real Monge-Ampère equations, we will add the following assumption \emph{on the particular choice of $j$, $\Delta$ and $G_1,\ldots,G_k$ considered}.

\begin{assumption}
\label{Jassumption}
We assume the following on $j$, $\Delta$ and $G_1,\ldots,G_k$:
\begin{enumerate}

\item \label{Jass_proper} $j_{\infty}+v_{\Delta}\geq \epsilon |x|$ for some $\epsilon>0$,
\item \label{Jass_translate} $j-j_{\infty}$ is bounded from below on $\mathring{\lalg{a}}^+$ 
\item $j_\infty$ is continuous and finite on $\lalg{a}^+$
\item \label{gass_integrable} $G_i^{-\epsilon}$ is integrable for some $\epsilon>0$. 
\end{enumerate}
\end{assumption}

Fix, for all $i\in\{1,\ldots,k\}$, a reference function $\fun{i}{\rf} \in \mathcal{P}(\Delta_i)$. 
We consider the continuity path of system of equations of the form 
\begin{equation} 
\label{eq:RealContPath}
\MAR{\fun{i}{t}}G_i(d\fun{i}{t}) = J \prod_m e^{-t\fun{m}{t}-(1-t)\fun{m}{\rf}} 
\end{equation}
where $0\leq t\leq 1$, and we are only interested in solutions $(\fun{i}{t})\in \prod_i \mathcal{P}(\Delta_i)$. Similarly as in Section~\ref{sec:intro_red}, we assume any solution to be normalized to satisfy 
$$\fun{1}{t}(0)=\fun{2}{t}(0)=\cdots=\fun{k}{t}(0).$$ 

Let $\barDH{i}\in \Delta_i$ denote the barycenter of $\Delta_i$ with respect to the measure with 
potential $G_i$
$$\barDH{i} = \int_{\Delta_i} p G_i(p)dp$$
and let $F_t$ denote the function on $\lalg{a}^+$ given by
$$
F_t = t\sum_{i=1}^k \barDH{i} +(1-t)v_{\Delta} +j_{\infty}.
$$
For $t\in [0,1]$ we will use $(\ddagger_t)$ to denote the following condition   
\begin{align*}
F_t(\xi) & \geq 0 \text{ for all }\xi\in \mathfrak{a}^+ \\
F_t(\xi) & = 0 \text{ only if }t=1, -\xi\in \mathfrak{a}^+ \text{ and } j_\infty(-\xi)=-j_\infty (\xi). 
\end{align*}
\begin{thm}
\label{thm:C0Estimates}
Let $t_0>0$ and $t\in (t_0,1]$. Assume $(\ddagger_t)$ is true.
Then there are $C^0$-estimates as in \eqref{eq:C0Assumption} for Equation~\eqref{eq:RealContPath} on $[t_0, t]$. 
Moreover, if $(\ddagger_t)$ is not true, then Equation~\eqref{eq:RealContPath} has no solution at $t$.
\end{thm}

\subsection*{Acknowledgements} 
It is a pleasure for the authors to thank the organizers of the Spring School  \emph{Flows and Limits in Kähler Geometry} in Nantes, part of a thematic semester of the Centre Henri Lebesgue, where the joint work leading to this article was initiated.
The first author is partially supported by the ANR Project FIBALGA ANR-18-CE40-0003-01. The second author was partially supported by the Olle Engkvist Foundation and the research council of Norway, grant number 240569.

\section{Reduction to $C^0$ estimates}
\label{sec:reductionC0}

Let $s\in [0,1]$ and $t\in [0,1]$. To prove Theorem~\ref{thm:ReductionToC0} we will consider the following more general version of \eqref{eq:ComplexContPath}:
\begin{equation}
    \label{eq:GenComplexContPath}
    e^{sh_i(f_i+V_i(\phi_i))}\omega_i^n = e^{-t\sum_{m=1}^k \phi_m}\theta_0^n
\end{equation}
where $\omega_i = \theta_i + \iddbar \phi_i$ and $\phi_i = \phi_i(s,t,\cdot)$. 

We will begin by proving that the set of $(t,s)\in [0,1)\times [0,1]$ such that \eqref{eq:GenComplexContPath} is solvable is open. Then we will establish a priori higher order estimates on solutions to \eqref{eq:GenComplexContPath} assuming a priori $C_0$-estimates. Finally we will prove a priori $C_0$-estimates for solutions to \eqref{eq:GenComplexContPath} when $(t,s)\in \{0\}\times [0,1]$. Using the Calabi-Yau Theorem we will produce a solution at $(t,s)=(0,0)$. Using this, together with the assumption in the Theorem~\ref{thm:ReductionToC0} on $C^0$-estimates for $(t,s)\in [0,1)\times \{1\}$, we will be able to conclude that \eqref{eq:GenComplexContPath} is solvable for $(t,s)\in (\{0\}\times [0,1])\cup ([0,1]\times \{1\})$. In particular, that \eqref{eq:GenComplexContPath} is solvable for $s=t=1$ proving Theorem~\ref{thm:ReductionToC0}.

\subsection{Openness}
We define the following Banach spaces 
$$ A= \left\{(\phi_i)\in \left(C^{4,\alpha}(X)\right)^k: \phi_i \textnormal{ is } K-\textnormal{invariant for all } i  \right\} $$
and 
$$ B = \left\{(v_i)\in \left(C^{2,\alpha}(X)\right)^k: v_i \textnormal{ is } K-\textnormal{invariant for all } i  \right\}. $$ 
Moreover, let $A_{(\theta_i)}$ be the set of $k$-tuples $(\phi_i)\in A$ such that $\omega_i:=\theta_i+\iddbar \phi_i>0$ for all $i$. Moreover, for each $i$, let $g_i$ denote the function
$$ g_i = h_i(f_i+V_i(\phi_i)). $$
Let 
$$ F:[0,1]\times  [0,1]\times A_{(\theta_i)} \rightarrow B$$ 
be defined by
$$
F(t,s,(\phi_i)) = 
\begin{pmatrix} 
\log\frac{\omega_1^n}{\theta_1^n} + g_1+{V_1}(\phi_1) + t\sum \phi_i - \hat \phi_1 \\
\vdots \\
\log \frac{\omega_k}{\theta_k^n} +g_k+{V_k}(\phi_k) + t\sum \phi_i - \hat \phi_k 
\end{pmatrix}
$$
where $\hat\phi_i$ is the average of $\phi_i$ with respect to $\theta_i$
$$ \hat\phi_i = \frac{\int_X \phi_i \theta_i^n}{\int_X \theta_i^n}. $$
Up to normalization, it follows that $F(t,s,(\phi_i))=0$ if and only if $(\phi_i)$ is a solution to \eqref{eq:GenComplexContPath} at $(t,s)$.
Moreover, in this case the measure
$$ \mu := e^{sg_i}\omega_i^n  $$ 
is independent of $i$.

To any Kähler metric $\omega$, expressed in local coordinates as 
$$\omega=\sqrt{-1}\sum \omega_{j,m} dz\wedge d\bar z,$$ 
and function $g$ we may associate a complex (1,0)-vector field 
$$ \frac{1}{\sqrt{-1}}\sum_{j,m} \frac{\partial g}{\partial \bar z_m} dz_j. $$ 
The following lemma characterizes the variation of $g_i$ with respect to $\phi_i$. 
\begin{lem}
\label{lemma:VariationOfExtraFactor}
Let $(v_i)\in A$. Then
\begin{equation} 
\left.\frac{dh_i(f_i+V_i(\phi_i+tv_i))}{dt} \right|_{t=0} = \left\langle\grad^\mathbb C_{\omega_i} h_i(f_i+V_i(\phi_i)),v_i\right\rangle. 
\label{eq:gVariation}
\end{equation}
\end{lem}
\begin{proof} 
By the chain rule, it suffices to prove this in the case $h_i=id$. In this case it is well known. However, for completeness we include a proof of it. By straight forward differentiation, the left hand side of \eqref{eq:gVariation} is $V_i(v_i)$.  
Moreover, by the definition of $f_i$
\begin{eqnarray} 
\iddbar \left( f_i + V_i(\phi_i)\right) & = & L_{V_i}(\theta_i) + L_{V_i}(\iddbar \phi_i) \nonumber \\
& = & L_{V_i}(\omega_i) \nonumber \\
& = & d(V_i\rfloor \omega_i). \nonumber
\end{eqnarray}
It follows that 
\begin{equation} \bar \partial \left( f_i + V_i(\phi_i)\right) - V_i\rfloor \omega_i \label{eq:holo10form} \end{equation}
is a $\partial$-closed (0,1)-form. Conjugating \eqref{eq:holo10form} gives a holomorphic $(1,0)$-form which, since $X$ is Fano, must vanish. Hence $\eqref{eq:holo10form}$ vanishes. Choosing coordinates $(z_1,\ldots,z_n)$ such that $ \omega_i = \sqrt{-1}\sum_j dz_j\wedge d z_j  $ and writing $V_i = \sum_j V_i^j \partial/\partial \bar z_j $ we get
\begin{eqnarray}
\sum_j \frac{\partial\left( f_i + V_i(\phi_i)\right)}{\partial\bar z_j} d\bar z_j & = & \bar \partial \left( f_i + V_i(\phi_i)\right) \nonumber \\
& = & V_i\rfloor \omega_i \nonumber \\
& = &
\left.\left(\sum_jV_i^j\frac{\partial}{\partial z_j}\right)\right\rfloor \left( \sqrt{-1}\sum_j dz_j\wedge d\bar z_j\right) \nonumber \\
& = & \sqrt{-1}\sum_j V_i^j d\bar z_j, \nonumber
\end{eqnarray} 
hence $\partial ( f_i + V_i(\phi_i))/d\bar z_j = \sqrt{-1}V_i^j$ for each $j$ and 
\begin{eqnarray}
\grad^\mathbb C_\omega\left( f_i + V_i(\phi_i)\right) & = & \frac{1}{\sqrt{-1}}\sum_j \frac{\partial\left( f_i + V_i(\phi_i)\right)}{d \bar z_j} \frac{\partial}{\partial z_j} \nonumber \\
 & = & \sum_j V_i^j \frac{\partial}{\partial z_j} \nonumber \\
 & = & V_i. \nonumber
\end{eqnarray} 
This proves the lemma. 
\end{proof}

Now, let $\Delta_{\omega_i,sg_i}$ be the $sg_i$-weighted Laplacian
$$ \Delta_{\omega_i,sg_i} = \Delta_{\omega_i} + \grad^\mathbb C_{\omega_i} sg_i $$
where $\Delta_{\omega_i}$ is the usual Laplace-Beltrami operator of the metric $\omega_i$. 
\begin{lem}
\label{lemma:Linearization}
The linearization of $F$ at $(t,s,\phi)$ with respect to the third argument is given by $H:A\rightarrow B$ defined by
\begin{equation}
\label{eq:Linearization}
H(v_1,\ldots,v_k) = \begin{pmatrix} 
-\Delta_{\omega_1,sg_1} v_1 + t\sum v_i - \hat v_1\\
\vdots \\
-\Delta_{\omega_k,sg_k} v_k + t\sum v_i - \hat v_k\\
\end{pmatrix}.
\end{equation}
where 
$$\hat v_i = \frac{\int v_i \theta_i^n}{\int\theta_i^n}.$$
Moreover, $H$ is elliptic. Finally, assume $F(t,\phi)=0$ and let $\langle\cdot,\cdot\rangle$ be the inner product on $(C^{2,\alpha}(X))^k$ given by
$$ \langle(u_i),(v_i)\rangle = \sum_i \int_X u_iv_i d\mu $$
Then $\langle H(u_1,\ldots,u_k) , (v_i)\rangle = \langle (u_i),H(v_1,\ldots,v_k) \rangle$ for any $(u_i),(v_i)\in (C^{2,\alpha})^k$.
\end{lem}
\begin{proof}
By Lemma~\ref{lemma:VariationOfExtraFactor}, the linearization of $g_i$ with respect to $\phi$ is
$\grad^\mathbb C_{\omega_i} g( u_i )$. Together with a standard computation (see for example the proof of Lemma~3 in \cite{Hul}), this proves the first part of the lemma. 

The second part of the lemma, i.e. that $H$ is elliptic, follows exactly as in the proof of Lemma 3 in \cite{Hul}.

The third part of the lemma is a consequence of standard properties for $\Delta_{\omega_i,sg_i}$, namely
\begin{equation} 
\label{eq:WeightedLaplacianPI}
\int_X (\Delta_{\omega_i,sg_i} u)v e^{sg_i}\omega_i^n = \int \langle du,dv \rangle_{\omega_i} e^{sg_i}\omega_i^n = \int_X u(\Delta_{\omega_i,sg_i} v) e^{sg_i}\omega_i^n
\end{equation}
for any $u,v\in C^{2,\alpha}$. 
From this, the third part of the lemma follows as in the proof of Lemma~3 in \cite{Hul}.
\end{proof}

\begin{lem}
\label{lemma:HInjective}
Assume $(t,s)\in [0,1)\times [0,1]$ and $(v_i)\in A$ satisfies for all $i$
\begin{equation}
    \Delta_{\omega_i,sg_i} v_i = \lambda\sum_{m=1}^k v_m
    \label{eq:LaplaceEigenfunction}
\end{equation}
for a $k$-tuple $\omega_1,\ldots,\omega_k$ satisfying for all $i$
\begin{equation} 
\Ric \omega_i - \sqrt{-1} \partial \bar \partial sg_i = \gamma + t\sum_{m=1}^k \omega_m + (1-t)\sum_{m=1}^k  \theta_m.
\label{eq:RicContinuityMethod}
\end{equation}
Then either $\lambda > t$ or $v_i$ is constant for all $i$.
\end{lem}
\begin{proof}
 The proof of this lemma follows closely the proof of Theorem~1.3 in \cite{Pin18} and the proof of Lemma~20 in \cite{Hul}. The crucial point that is new in this setting is the following Weitzenb\"ock identity:
 \begin{equation}
     \label{eq:WeitzenbockId}
     \int_X \left\langle d\left(\Delta_{\omega_i,sg_i} u\right), du\right\rangle_{\omega_i} e^{g_i}\omega_i^n \geq \int_X \left(\Ric_{\omega_i} - \sqrt{-1} \partial \bar \partial g_i\right) \left(\grad^\mathbb C_{\omega_i} u,\overline{\grad^\mathbb C_{\omega_i} u}\right) e^{g_i}\omega_i^n
 \end{equation}
 valid for any $u\in C^{4,\alpha}$.
 Using \eqref{eq:WeitzenbockId} and following the argument in \cite{Hul} or \cite{Pin18} proves the lemma. We will now prove \eqref{eq:WeitzenbockId}. In the following computation we will suppress the index on $\omega_i$ and $g_i$. Instead, lower index $i,\bar i,j,\bar j,p$ and $ \bar p$ will denote covariant differentiation. We will also suppress summation symbols.  
 \begin{eqnarray}
 \int_X (\Delta_\omega u)_i u_{\bar i} e^{g}\omega^n & = &  \int_X u_{j\bar j i}u_{\bar i}e^{g}\omega^n \nonumber \\
 & = & \int_X \left(u_{ij\bar j}-R^p_{ji\bar j}\right)u_{\bar i}e^{g}\omega^n \nonumber \\
 & = & \int_X \Ric_{i\bar p}u_pu_{\bar i}e^{g}\omega^n + \int_X u_{ij}u_{\bar i\bar j}e^{g}\omega^n - \int_X u_{ij}u_{\bar i}g_{\bar j}e^{g}\omega^n. \nonumber 
 \end{eqnarray}
 Moreover,
 \begin{eqnarray}
    \int_X \left(\grad^\mathbb C_{\omega} g (u)\right)_iu_{\bar i} e^{g}\omega^n 
    & = & \int_X \left(u_jg_{\bar j}\right)_iu_{\bar i} e^{g}\omega^n \nonumber \\
    & = & \int_X u_{ji}g_{\bar j}u_{\bar i} e^{g}\omega^n 
    + \int_X u_j g_{\bar ji} u_{\bar i} e^{g}\omega^n \nonumber 
 \end{eqnarray}
 We get
 \begin{eqnarray}
 \int_X  (\Delta_\omega u + \grad^\mathbb C_{\omega} g (u))_i u_{\bar i} e^{g}\omega^n & = & \int_X \left(\Ric_{i\bar p}u_pu_{\bar i}-u_j g_{\bar ji} u_{\bar i}\right)e^{g}\omega^n \nonumber \\
 & & + \int_X u_{ij}u_{\bar i\bar j}e^{g}\omega^n, \nonumber
 \end{eqnarray}
 and \eqref{eq:WeitzenbockId} follows. 
\end{proof}

\subsection{Estimate on $\Delta_{\theta_i}\phi_i$}
\begin{lem}
\label{lemma:LaplacianEstimate}
Assume $(\phi_i)$ satisfies $\eqref{eq:GenComplexContPath}$ for some $(t,s)\in [0,1]\times [0,1]$. Then  
$$ \max_i \sup_X |\Delta_{\theta_i}\phi_i| \leq C $$
where $C$ depends only on $\max_i ||\phi_i||_{C^0(X)}$.
\end{lem}
\begin{proof}
The proof follows the standard method of Yau. In particular, see \cite{Pin18} and \cite{Hul} where the cases of coupled K\"ahler-Einstein metrics and coupled K\"ahler-Ricci solitons is treated. The method is based on expanding the quantity 
$$ \Delta_{\omega_i} \left(e^{-C_1\phi_i}(n+\Delta_{\theta_i}\phi_i)\right) $$ at a point where $e^{-C_1\phi_i}(n+\Delta_{\theta_i}\phi_i)$ attains its maximum. To prove the lemma, some extra care need to be taken when estimating 
\begin{eqnarray} \Delta_{\theta_i}g_i & = & \Delta_{\theta_i}\left(h_i(f_i+V_i(\phi_i))\right) \nonumber \\
& = & h_i'\Delta_{\theta_i} \left(f_i+V_i(\phi_i)\right) + h_i''\left|d(f_i+V_i(\phi_i))\right|^2_{\theta_i}. \label{eq:LaplaceOfgi}
\end{eqnarray}
where $h_i'$ and $h_i''$ are the first and second derivatives of $h_i$ at the point given  by $f_i+V_i(\phi_i)$. Now, by Corollary~5.3, page 768 in \cite{Zhu00}, $|f_i+V_i(\phi_i)|$ can be bounded by a constant independent of $\omega_i$. It follows that $h'_i$ is bounded. Moreover, $|\Delta_{\theta_i}(f_i+V_i(\phi_i)) |$ can be bounded as in \cite{Zhu00}, page 769. Finally, by concavity of $h_i$, the second term is bounded from above by 0. Apart from this, the proof follows the arguments in \cite{Pin18} and \cite{Hul}.
\end{proof}

\subsection{$C^0$-estimates for $t=0$}
\begin{lem}
\label{lemma:C0Estimatet=0}
There is a constant $C$ such that any solution $(\phi_i)$ to \eqref{eq:GenComplexContPath} at $t=0$ and $s\in [0,1]$
satisfies
\begin{equation}
    \label{eq:C0Estimatet=0}
    \sup_X \left|\phi_i-\sup_X\phi_i\right|<C
\end{equation}
for all $i$. 
\end{lem}
\begin{proof}
Let 
$$ F^i_s = -s g_i + \log \frac{\theta_0^n}{\theta_i^n}+c_s $$
where $c_s$ is a constant such that 
$$ \int_X e^{F^i_s}\theta_i^n = \int_X \theta_i^n. $$
This means each $\phi_i$ satisfies the equation
$$ \omega_i^n = e^{F^i_s}\theta_i^n. $$

As above, we note that by Corollary 5.3, page 768 in \cite{Zhu00}, $ |g_i| = |h_i(f_i+V_i(\phi_i))| $ can be bounded by a constant independent of $\phi_i$. Moreover, since $\int \omega_i^n=\int \theta_i^n$
\begin{eqnarray}
 |c_s| &  = & \left|\log\int_X e^{-s g_i}\theta_0^n - \log\int_X \theta_i^n\right| \nonumber \\
 & \leq & s\sup_X |g_i| +\left|\log\int_X \theta_0^n - \log\int_X \theta_i^n\right| \nonumber \\
 & = & s\sup_X |g_i| + c \nonumber
 \end{eqnarray}
where $c$ is a constant independent of $\phi_i$. It follows that $|F_s^i|$ can be bounded by a constant independent of $\phi_i$. By invariance of \eqref{eq:GenComplexContPath} at $t=0$ we may assume $\sup \phi_i = -1$ for all $i$. Applying the argument in \cite{Tia96}, page 157-159, we get a uniform constant $C$ satisfying \eqref{eq:C0Estimatet=0}.
\end{proof}

\subsection{Proof of Theorem~\ref{thm:ReductionToC0}}
\begin{proof}[Proof of Theorem~\ref{thm:ReductionToC0}]
First of all, Lemma~\ref{lemma:Linearization} describes the linearization $H$ of $F$. By Lemma~\ref{lemma:HInjective} the kernel of $H$ is trivial. A standard application of the Implicit Function Theorem then shows that the set of $(t,s)$ such that \eqref{eq:GenComplexContPath} is solvable is open in $[0,1)\times [0,1]$. Moreover, by the Calabi-Yau Theorem we may find $\phi_1,\ldots,\phi_k$ solving \eqref{eq:GenComplexContPath} at $(t,s)=(0,0)$. Since $\theta_0$ is $K$-invariant, $\phi_i$ will be $K$-invariant for each $i$. We will now consider the path 
$$(t,s)\in A:=\left(\{0\}\times [0,1]\right) \cup \left([0,1]\times \{1\}\right) \subset [0,1]\times [0,1].$$
In other words, we consider $(s,t)$ such that either $t=0$ or $s=1$. By the assumed $C^0$ estimates in \eqref{eq:C0Assumption}, Lemma~\ref{lemma:LaplacianEstimate} and Lemma~\ref{lemma:C0Estimatet=0} any solution to \eqref{eq:GenComplexContPath} has bounded $C^0$ norm and bounded Laplacian. It then follows from Theorem~1 in \cite{Wan12} that the $C^{2,\alpha}$-norm $\sup_i ||\phi_i||_{C^{2,\alpha}}$ of any solution to \eqref{eq:ComplexContPath} along this path is uniformly bounded. We concluded that the set of $(t,s)\in A$ such that \eqref{eq:ComplexContPath} is solvable is open and closed in $A$, hence \eqref{eq:ComplexContPath} is solvable at $s=1$ for any $t\in [0,1]$. This proves the theorem. 
\end{proof}

\section{$C^0$-estimates for coupled real Monge-Ampère equations on polyhedral cones}
\label{sec_C0}

In this section we will prove Theorem~\ref{thm:C0Estimates}. The proof essentially follows the argument in \cite{DelKE} with adaptations for the coupled case from \cite{Hul}. Both of these are in turn generalizations of the argument in \cite{WZ04}. 

Set 
$$\nu_t:=t\sum_m\fun{m}{t}+(1-t)\sum_m\fun{m}{\rf}+j$$
on $\mathring{\mathfrak{a}}^+$. 
Note that $\nu_t$ is smooth, strictly convex. Moreover, since $u_m^t, u_m^{\rf}\in \mathcal P(\Delta_m)$ for each $m$, we get that $\nu_t-j\in\Potent{}$, hence by Assumption~\ref{Jassumption}(\ref{Jass_proper}) that $\nu_t(x)$ goes to $+\infty$ whenever $x$ goes to infinity. 
As a consequence, the following real number $m_t\in \bbR$, point $x_t\in\mathring{\lalg{a}}^+$ 
and bounded convex set $A_t\subset \mathring{\lalg{a}}^+$ are well-defined: 
\[ m_t=\min \nu_t = \nu_t(x_t) \qquad A_t := 
\{x\in \mathring{\lalg{a}}^+ ~;~ m_t\leq \nu_t(x) \leq m_t+1 \}. \]

To prove Theorem~\ref{thm:C0Estimates} it will suffice with a weaker version of Assumption~\ref{Jassumption}(\ref{Jass_translate}). To state this weaker assumption we define, for every positive $M\in \bbR$, the following subset of $\lalg{a}^+$
$$ K_M = \left\{x\in \lalg{a}^+;\; \exists u\in \mathcal{P}(\Delta),\, j(x) + u(x)  - \inf_{\lalg{a}^+} (j+u) \leq M \right\}. $$
The significance of these sets is that $\nu_t>M$ on $\lalg{a}^+\setminus K_M$ and instead of Assumption~\ref{Jassumption}(\ref{Jass_translate}) we will use the following weaker assumption:
\begin{itemize}
\item  For any $M>0$, $j-j_{\infty}$ is bounded from below on $K_M$.
\end{itemize}

For a function $g$ in $\lalg{a}^+$ and $x,\xi\in \lalg{a}^+$, we will use the notation $d_x g(\xi)$ to denote the directional derivative of $g$ at the point $x$ in the direction $\xi$. As in \cite{WZ04}, the first step of the proof of Theorem~\ref{thm:C0Estimates} is to reduce $C^0$-estimates for $u_i^t$ to estimates on $\nu_t$. 
\begin{lem}[Reduction to estimates on $|m_t|$, $|x_t|$ and linear growth of $\nu_t$]
\label{lem:RedToEstOnNu}
Assume $t\geq t_0>0$ and the following hold:
\begin{eqnarray}
|m_t| & \leq & \C, \label{eq:mt_bound} \\  
\nu_t(x) & \geq & \C|x-x_t|-\C \label{eq:lin_growth} \\
|x_t| & \leq & \C. \label{eq:xt_bound}
\end{eqnarray}
Then 
\[ \max_i \sup_{\lalg{a}^+} |\fun{i}{t}-\fun{i}{\rf}| \leq \C. \]
\end{lem}

\begin{proof}
We follow the argument in \cite{Hul}. There is only one added difficulty which is 
solved by using Assumption~\ref{Jassumption}(\ref{gass_integrable}). 
To prove the lemma, first recall that the Legendre transform $u\mapsto \mathcal{L}(u)$ 
is such that, for $u, v \in \mathcal{P}(\Delta)$, 
$\sup_{\lalg{a}^+} |u-v| = \sup_{\Delta} |\mathcal{L}(u)-\mathcal{L}(v)|$. 
Furthermore, if $u\in \Potent{}$ then $\mathcal{L}(u)$ is bounded on $\Delta$, 
hence to get the conclusion it is enough to prove estimates on 
$\mathcal{L}(\fun{i}{t})$ for each $i$. 

For each $i$, let $w_i^t=\mathcal{L}(\fun{i}{t})$. We may first obtain a bound on $w_i^t$ at one given point. 
Note that $w_i^t(d\fun{i}{t}(0))=-\fun{i}{t}(0)$. Furthermore, 
$\fun{i}{t}(0)= \frac{1}{k}\sum_l\fun{l}{t}(0)$ by the normalization we chose. 
Since $\fun{l}{t}\in\Potent{l}$ for all $l$, and $|x_t|$ is controlled,  
$|\fun{i}{t}(0)-\frac{1}{k}\sum_l\fun{l}{t}(x_t)|$ is also controlled. 
Moreover, we claim that $|j(x_t)|$ is uniformly bounded. To see this, let $\gamma\in \mathring{\lalg{a}}$ and note that since $d_{x_t}j \in \Delta$ we have that $x_t$ is in the compact set 
\begin{equation}
\label{eq:xtset}
    B(0,C_4)\cap \{dj(\gamma)\geq -v_\Delta(\gamma)\}. 
\end{equation}
Since $j\rightarrow +\infty$ as we approach the boundary of $\lalg{a}^+$ we get that \eqref{eq:xtset} is included in the interior of $\lalg{a}^+$. This means $j$ is bounded on \eqref{eq:xtset}, proving the claim. Moreover, since $|j(x_t)|$ is bounded we get that $|\fun{i}{t}(0)-m_t|$, and hence $|\fun{i}{t}(0)|$ is bounded. 
Let now $\hat{w}_i^t$ denote the average of $w_i^t$ on $\Delta_i$. 
The conclusion of the Lemma will follow from using  
\begin{align*}
\sup_{\Delta_i} |w_i^t(p)-w_i^t(0)| & 
\leq \C|w_i^t-\hat{w}_i^t|_{1-r/s} \\
& \leq \C|dw_i^t|_{L^s(\Delta_i)}
\end{align*}
for $s>r$ (recall that $r$ is the dimension of $\mathfrak{a}$), 
where the second inequality holds by Sobolev inequalities.
Using Assumption~\ref{Jassumption}(\ref{gass_integrable}), 
we choose $p>1$ such that $G_i^{-1/(p-1)}$ is integrable. Thanks to Hölder's Inequality, we have 
\begin{align*}
\int_{\Delta_i} |dw_i^t|^s & = 
\int_{\Delta_i} (|dw_i^t|^sG_i^{1/p})G_i^{-1/p} \\
& \leq \left(\int_{\Delta_i} |dw_i^t|^{ps}G_i\right)^{1/p}  \left(\int_{\Delta_i} G_i^{-1/(p-1)}\right)^{(p-1)/p} \\
& \leq \C\left(\int_{\Delta_i} |dw_i^t|^{ps}G_i\right)^{1/p}
\end{align*}
To finally bound $|dw_i^t|_{L^s(\Delta_i)}$, we Legendre transform back 
to use the equation:
\begin{align*}
\int_{\Delta_i} |dw_i^t|^{ps}G_i & = 
\int_{\lalg{a}^+} |x|^{ps}G_i(d\fun{i}{t})\MAR{\fun{i}{t}} \\
& = \int_{\lalg{a}^+} |x|^{ps}e^{-\nu_t(x)} \\
& \leq C
\end{align*}
thanks to the uniform linear growth estimate on $\nu_t$. 
\end{proof}

The first step in proving the estimates \eqref{eq:mt_bound}, \eqref{eq:lin_growth} and \eqref{eq:xt_bound} on $\nu_t$ is to get bounds on $\mathrm{Vol}(A_t)$. 
\begin{lem}
There exists a constant $\Cl{lowvol}>0$ independent of $t$ such that 
\[ \Cr{lowvol}\leq \mathrm{Vol}(A_t). \]
\end{lem}
\begin{proof}
We have $d\nu(\xi) < j_\infty(\xi) + v_\Delta(\xi)$ which is bounded on $\lalg{a}^+$ by continuity of $j_\infty$. This means $\nu_t<m_t+C|x-x_t|$ on $(x_t+\lalg{a}^+)$ for some constant C>0, hence the set $(x_t+\lalg{a}^+)\cap B(x_t,1/C)$ is included in $A_t$ for some uniform $r>0$. This gives the desired volume bound.  
\end{proof}

On the other hand, there is an upper bound in terms of $m_t$:
\begin{lem}
There exists a constant $\Cl{upvol}$ such that, for $t\geq t_0$,  
\[ \mathrm{Vol}(A_t)\leq \Cr{upvol}e^{m_t/2}. \]
\end{lem}
\begin{proof}
The proof follows exactly that of Wang and Zhu \cite{WZ04} and Donaldson \cite{Don08}. The main point is the following convex geometric fact: If $f$ is a convex function on $\mathbb R^n$ attaining its minimal value such that $\det(d^2f)$ is bounded from below by some $\lambda>0$ on $\{\inf f \leq f\leq \inf f+1\}$, then 
$$\Vol\{\inf f \leq f\leq \inf f+1\}\leq \Cl{MABound}\lambda^{-1/2},$$
where $\Cr{MABound}$ only depends on the dimension $n$ (see Proposition~2 in \cite{Don08}). 
The lemma then follows by observing that
\[ \MAR{\nu_t} \geq \MAR{tu_i}\geq t_0^n\Cl{comparison}e^{-m_t} \]
for some constant $\Cr{comparison}$ which does not depend on $t\geq t_0$. 
This last inequality is obtained from \eqref{eq:RealContPath}
by noting that $G_i$ is continuous positive on $\mathring{\Delta}_i$, hence $1/G_i\geq c>0$ 
for some constant $c$. 
\end{proof}

From this it follows, following the proof of Wang and Zhu or Donaldson with only notational changes:
\begin{lem}[Estimates on $m_t$, and on linear growth of $\nu_t$]
\label{lem_mlinear}
\mbox{}
\begin{itemize}
\item There exists a constant $C$ independent of $t$, such that $|m_t|\leq C$.
\item There exist a constant $\kappa >0$ and a constant $C$, both independent of $t$, such that 
for $x\in \mathfrak{a}^+$,
\[
\nu_t(x) \geq \kappa |x-x_t|-C.
\]
\end{itemize}
\end{lem}

Recall that $(\ddagger_t)$ denotes the condition that
$$ F_t(\xi) = t\sum_i\barDH{i}(\xi) +(1-t)v_{\Delta}(\xi) +j_{\infty}(\xi) \geq 0 $$
with equality only if $t=1$, $-\xi\in \mathfrak{a}^+$ and $j_\infty(-\xi)=-j_\infty (\xi)$. By Lemma~\ref{lem:RedToEstOnNu} and Lemma~\ref{lem_mlinear}, the $C^0$-estimates of Theorem~\ref{thm:C0Estimates} are reduced to proving a bound on $|x_t|$. To prove the first part of Theorem~\ref{thm:C0Estimates} it thus suffices to rule out the case $|x_t|\rightarrow \infty$ whenever $(\ddagger_t)$ holds. Most of the rest of this section is dedicated to this. However, before we start we prove the second point of Theorem~\ref{thm:C0Estimates}, namely that if there is a solution to \eqref{eq:RealContPath} at $t$ then $(\ddagger_t)$ holds.
\begin{lem}[Obstruction]
\label{lem_obstruction}
Assume there exists a solution in $\prod_i\Potent{i}$
to \eqref{eq:RealContPath} at time $t$. Then $(\ddagger)_t$ hold. 
\end{lem}

\begin{proof}
The proof is based on the vanishing of
\begin{equation}
\label{eq:stokes}
0 = \int_{\mathfrak{a}^+} d\nu_t(\xi) e^{-\nu_t},
\end{equation}
for any $\xi \in \mathfrak{a}$, which follows from integration by parts, Assumption~\ref{Jassumption}(\ref{Jass_proper}) and the the fact that $J=0$ on $\partial \lalg{a}^+$.   
(see \emph{e.g.} \cite[Section~5]{DelKE} for details).
By definition of $\nu_t$, the right hand side of \eqref{eq:stokes} splits as a sum 
\begin{equation} 
\label{eq:stokesexpanded}
0 = t\sum_m\int_{\mathfrak{a}^+} d\fun{m}{t}(\xi)e^{-\nu_t} 
+(1-t)\sum_m\int_{\mathfrak{a}^+} d\fun{m}{\rf}(\xi)e^{-\nu_t}
+ \int_{\mathfrak{a}^+} d j(\xi)e^{-\nu_t}. 
\end{equation}
Note that, by definition of $\barDH{m}$, the fact that $G_m$ has unit mass on $\Delta_m$ 
and \eqref{eq:RealContPath}, 
we have 
\begin{equation} 
\label{eq:bar}
\int_{\mathfrak{a}^+} d\fun{m}{t}(\xi)e^{-\nu_t}=\barDH{m}(\xi).
\end{equation}
For $\xi\in\mathfrak{a}^+$, we have 
\begin{equation} 
\label{eq:dj}
d j(\xi)\leq j_{\infty}(\xi) 
\end{equation}
and, by strict convexity of $\fun{m}{\rf}$,
\begin{equation}
    \label{eq:duref}
d\fun{m}{\rf}(\xi)< v_{\Delta_i}(\xi).
\end{equation}
This means $F_t(\xi)$ is bounded from below by the right hand side of \eqref{eq:stokesexpanded}. In other words $F_t(\xi)\geq 0$. Moreover, by \eqref{eq:duref}, this bound is strict whenever $t<1$. Finally, equality in \eqref{eq:dj} only holds if $j$ is affine on the subspace of $\mathbb R^n$ generated by $\xi$. This can only be true if $-\xi\in \mathfrak{a}^+$ and $j_\infty(-\xi)=-j_\infty(\xi)$. We conclude that $F_t(\xi)=0$ only if $t=1$, $-\xi\in \mathfrak{a}^+$ and $j_\infty(-\xi)=-j_\infty(\xi)$, hence $(\ddagger)_t$ is satisfied.
\end{proof}

To conclude the proof of Theorem~\ref{thm:C0Estimates}, we assume that the $C^0$-estimates fail on some interval 
$[t_0,t']\subset [0,1]$. By Lemma~\ref{lem:RedToEstOnNu} and Lemma~\ref{lem_mlinear} it follows that we may pick a converging sequence $t_i$ of elements in $[t_0,t']$, whose limit 
we denote by $t_{\infty}$, such that 
$\lim |x_{t_i}|=\infty$ and $\xi_i:=x_{t_i}/|x_{t_i}|$ admits a limit 
$\xi_{\infty}\in \mathfrak{a}^+$. We will prove that $(\ddagger_{t_\infty})$ does not hold and hence, by monotonicity of $F_t$ with respect to $t$, that $(\ddagger_{t'})$ does not hold. This is essentially the content of the following lemma:
\begin{lem}
\label{lem_limits}
Under the assumptions in the preceding paragraph
$$ F_{t_\infty}(\xi_\infty) = \big(t_{\infty}\sum_m \mathrm{bar}_m +(1-t_{\infty})v_{\Delta}+j_{\infty}\big)(\xi_{\infty}) = 0. $$
Moreover, if $t_\infty = 1$ then $-\xi_\infty\notin \lalg{a}^+$ or $j_\infty(-\xi_\infty)\not=-j_\infty(\xi_\infty)$.
\end{lem}
We abbreviate the indices $t_i$ by $i$ from now on. The proof again relies on the vanishing \eqref{eq:stokes}, with $\xi=\xi_i$. Using this, it will follow from \eqref{eq:bar} and the following lemma:
\begin{lem}
\label{lem:conv}
\begin{eqnarray}
\int_{\lalg{a}^+} d j(\xi_i) e^{-\nu} & \rightarrow & j_\infty(\xi_\infty) \\
\int_{\lalg{a}^+} d \fun{m}{\rf}(\xi_i) e^{-\nu} & \rightarrow & v_{\Delta_m}(\xi_\infty)
\end{eqnarray}
\end{lem}
To prove Lemma~\ref{lem:conv} we will first identify suitable sets on which $e^{-\nu_i}$ is concentrated. 
Let $U_{R,M} = B(x_i,R)\cap K_M$.
\begin{lem}
\label{lem:consmass}
Let $\kappa$ be as in Lemma~\ref{lem_mlinear}. Then
$$ \int_{\lalg{a}^+\setminus U_{R,M}} e^{-\nu} \leq e^{-m_t-\min\{M,\kappa R\}}R^n|B(0,1)| $$
and 
$$ \int_{\partial U_{R,M}} e^{-\nu} \leq e^{-m_t-\min\{M,\kappa R\}}R^{n-1}|\partial B(0,1)|. $$
\end{lem}
\begin{proof}
First of all, note that the complement of $K_M$ in $\mathring{\lalg{a}}^+$ is given by
$$ \mathring{\lalg{a}}^+\setminus K_M = \{x \in \lalg{a}^+; \; \forall u\in \mathcal{P},\, j(x)+u(x) -\inf_{\lalg{a}^+} (j+u) > M\}. $$
In particular, since 
$$\sum_m tu_m+(1-t)\sum_m\fun{m}{\rf}\in \mathcal{P}(\Delta)$$
we get that 
$$ \nu = j+\sum_m tu_m+(1-t)\sum_m\fun{m}{\rf} > \inf \nu + M $$
outside $K_M$. The lemma follows from this and the fact that, by Lemma~\ref{lem_mlinear}, $\nu\geq\kappa R+m_t$ on $\lalg{a}^+\setminus B(x_t,R)$.
\end{proof}
We also get the following bound:
\begin{lem}
\label{lem:boundgrad}
For any $\xi\in \mathfrak{a}$,
$$ \left|\int_{U_{M,R}} d j(\xi) e^{-\nu}\right| < C|\xi| $$
where $C$ is independent of $u$.
\end{lem}
\begin{proof}
By Stokes' theorem and \eqref{eq:stokesexpanded},
\begin{eqnarray}
\left|\int_{U_{M,R}} d j(\xi) e^{-\nu}\right| & = & \left|\int_{U_{M,R}} d \nu(\xi) e^{-\nu} - \int_{U_{M,R}} d \left(\sum_m tu_m+(1-t)\sum_m\fun{m}{\rf}\right)(\xi) e^{-\nu}\right| \nonumber \\
& \leq & \left|\int_{\partial U_{M,R}}  n(\xi) e^{-\nu}\right|  + \left|\int_{U_{M,R}} d \left(\sum_m tu_m+(1-t)\sum_m\fun{m}{\rf}\right)(\xi) e^{-\nu}\right|  \label{eq:StokesEq}
\end{eqnarray}
where $n$ is a unit length inward pointing normal on $\partial U_{M,R}$. As $|n(\xi)|\leq |\xi|$ and 
$$d \left(\sum_m tu_m+(1-t)\sum_m\fun{m}{\rf}\right)\in \Delta,$$ 
noting that $\Delta$ is contained in some ball centered at the origin of radius $\Cr{DeltaBound}>0$, we get that \eqref{eq:StokesEq} is bounded by
\begin{eqnarray}
 |\xi|\int_{\partial U_{M,R}}   e^{-\nu}  + \Cl{DeltaBound}|\xi| \int_{U_{M,R}}  e^{-\nu} & \leq & \Cl{BoundTerm} |\xi|+\Cr{DeltaBound}|\xi| \nonumber
\end{eqnarray}
where the last inequality is a consequence of Lemma~\ref{lem:consmass} and the fact that $$\int_{\lalg{a}^+} e^{-\nu}=1.$$
\end{proof}
We are now ready to prove Lemma~\ref{lem:conv}. 
\begin{proof}[Proof of Lemma~\ref{lem:conv}]
As in \cite{DelKE}, let $\gamma$ be a point in $\lalg{a}^+$ such that the closure of $ B(\gamma,R) $ is contained in the interior of $\lalg{a}^+$. By assumption, $j \geq j_\infty + C$ on $U_{M,R}$ for some $C\in \bbR$. Let $\xi_i'=(x_i-\gamma)/|x_i|$ and $\xi_i'' = \gamma/|x_i|$. It follows by convexity of $j$ that for any $y \in B(0,R)$ such that $x_i+y\in U_{M,R}$
\begin{eqnarray}
d_{x_i+y}j(\xi_i') & \geq & \frac{j(x_i+y)-j(\gamma+y)}{|x_i-\gamma|} \nonumber \\
& \geq & \frac{j_\infty(x_i+y)+C-j(\gamma+y)}{|x_i-\gamma|} \nonumber \\
& \geq & j_\infty\left(\xi_i+\frac{y}{|x_i|}\right)\frac{|x_i|}{|x_i-\gamma|} -\frac{C+\sup_{B(\gamma,R)} |j|}{|x_i-\gamma|} \nonumber \\
& \rightarrow & j_\infty(\xi_\infty)  \label{eq:gradconv}
\end{eqnarray} 
uniformly in $y$ as $|x_i|\rightarrow \infty$ since $j_\infty$ is continuous on $\lalg{a}^+$ and $\xi_i+y/|x_i|\rightarrow \xi$. 

By Lemma~\ref{lem:boundgrad} 
$$
\left|\int_{U_{M,R}} d j(\xi_\infty) e^{-\nu} dx - \int_{U_{M,R}} d j(\xi_i') e^{-\nu} dx  \right| \leq C|\xi_i-\xi_\infty|
$$ 
which vanishes as $i\rightarrow \infty$. We conclude that 
$$ \int_{U_{M,R}} d j(\xi_\infty) e^{-\nu} dx \geq j_\infty(\xi_\infty)\int_{U_{M,R}} e^{-\nu} dx - o(1) $$
where $o(1)\rightarrow 0$ as $i\rightarrow \infty$. Moreover, by standard convexity properties, $d j(\xi_\infty)\leq j_\infty(\xi)$ everywhere on $\lalg{a}^+$. This proves the first point in the lemma. The second point in the lemma follows by a similar argument. 
\end{proof}

We will now prove Lemma~\ref{lem_limits}.
\begin{proof}[Proof of Lemma~\ref{lem_limits}]
By \eqref{eq:stokesexpanded}
$$ 0 = t_i\sum_m\int_{\mathfrak{a}^+} d\fun{m}{i}(\xi_i)e^{-\nu_i} 
+(1-t_i)\sum_m\int_{\mathfrak{a}^+} d\fun{m}{\rf}(\xi_i)e^{-\nu_i}
+ \int_{\mathfrak{a}^+} d j(\xi_i)e^{-\nu_i}. $$
for each $i$. By Lemma~\ref{lem_limits} and \eqref{eq:bar}, the right hand side of this converges to $F_{t_\infty}$, hence $F_{t_\infty}=0$. Thus, to conclude the lemma we just need to rule out the case that $t_\infty = 1$, $-\xi_\infty\in \lalg{a}^+$ and $j_\infty(-\xi_\infty) = - j_\infty(\xi_\infty)$. By the latter of these conditions we have that $dj(\xi_\infty)$ is constant, hence
$$ \int_X dj(\xi_\infty) e^{-\nu_i} = j_\infty(\xi_\infty) $$
for all $i$. Since
$$ F_1(\xi_\infty) = \big(\sum_m \mathrm{bar}_m +j_{\infty}\big)(\xi_{\infty}) = 0 $$
we may plug in $\xi=\xi_\infty$ and $t=t_i$ into \eqref{eq:stokesexpanded} to get
$$ 0 = (1-t_i)\left(j_\infty(\xi_\infty) + \int_{\lalg{a}^+} \sum_m d\fun{m}{\rf}(\xi_\infty)e^{-\nu_i}\right). $$
This means 
$$ \int_{\lalg{a}^+} \sum_m d\fun{m}{\rf}(\xi_\infty)e^{-\nu_i} = -j_\infty(\xi_\infty) $$ 
for all $i$. However, since the right hand side of this is independent of $i$ and the left hand side is strictly less than $v_\Delta(\xi_\infty)$ and converging towards $v_\Delta(\xi_\infty)$ we get a contradiction. 
\end{proof}

\begin{proof}[Proof of Theorem~\ref{thm:C0Estimates}]
The first part of the theorem is given by Lemma~\ref{lem_limits}. The second part of the theorem is given by Lemma~\ref{lem_obstruction}.
\end{proof}

\section{Case of Horosymmetric manifolds} 
\label{sec:horosym}

In this section we will recall the definition of horosymmetric manifolds, and the tools developed in \cite{DelHoro} to study their Kähler geometry. We will then use these tools to translate \eqref{eq:ComplexContPath} in the horosymmetric setting to a real Monge-Ampère equation (see Theorem~\ref{thm:translation_horo} below). Combining with  Theorem~\ref{thm:ReductionToC0} and Theorem~\ref{thm:C0Estimates}, we obtain the proof of Theorem~\ref{thm:cpld_can_horo}. 
Throughout this section, if a Lie group is denoted by an upper-case character $G$, $T$, etc, its Lie algebra will be denoted by the corresponding lower case \emph{fraktur} character $\lalg{g}$, $\lalg{t}$, etc.

\subsection{Setting}

Let $G$ denote a connected complex linear Lie group, and assume that the 
Fano manifold $X$ is equipped with an action of $G$. 
We assume that there exists a point $x$ in $G$, a parabolic subgroup $P$ of $G$, 
a Levi subgroup $L$ of $P$ and a complex Lie algebra involution $\sigma$ of the Lie 
algebra of $L$ such that $G\cdot x$ is open and dense in $X$, and
the Lie algebra of the isotropy group $G_x$ of $x$ is the direct sum of the Lie algebra 
of the unipotent radical of $P$ and of the Lie subalgebra of the Lie algebra of $L$ 
fixed by $\sigma$. In other words, the manifold $X$ is \emph{horosymmetric} as 
introduced in \cite{DelHoro}. We fix such a point $x$ and denote by $H$ the group $G_x$.

Let us fix $T_s$ a torus in $L$, maximal for the property that $\sigma$ 
acts on $T_s$ as the inverse. Let $T$ be a $\sigma$-stable maximal torus of 
$L$ containing $T_s$, $\Phi$ the root system of $G$ with respect to $T$ and 
$\Phi_L$ that of $L$. We let $Q$ denote the Borel subgroup opposite to $P$, 
$Q^u$ denote its unipotent radical, and $\Phi_{Q^u}$ the roots of $Q^u$. 
Let $B$ denote a Borel subgroup of $G$, such that $T\subset B\subset Q$, 
with corresponding positive roots $\Phi^+$, and such that for any 
$\beta\in \Phi_L^+=\Phi_L\cap \Phi^+$, either $\sigma(\beta)=\beta$ or 
$-\sigma(\beta)\in \Phi_L^+$. 
We further introduce the notation $\Phi_s=\Phi_L\setminus \Phi_L^{\sigma}$.

Given $\beta\in \Phi_s$, we define the associated \emph{restricted root} 
$\bar{\beta}$ by $\bar{\beta}=\beta-\sigma(\beta)$, 
and let $\bar{\Phi}$ denote the set of all restricted roots, 
called the \emph{restricted root system}. To any restricted root $\alpha$ 
is associated an integer $m_{\alpha}$, its \emph{multiplicity}, equal to the number of 
$\beta\in \Phi_s$ such that $\bar{\beta}=\alpha$.

We fix $\theta$ a Cartan involution of $G$ commuting with $\sigma$ 
and $K=G^{\theta}$ the corresponding maximal compact subgroup.
Let $\lalg{a}_s$ denote the real vector space $i\lalg{k}\cap \lalg{t}_s$.
It may naturally be identified with $\mathfrak{Y}(T_s)\otimes \bbR$, 
where $\mathfrak{Y}(T_s)$ denotes the set of one parameter subgroups of $T_s$, 
or also with $\mathfrak{Y}(T/(T\cap H))\otimes \bbR$, since 
$T_s\rightarrow T/(T\cap H)$ is an isogeny.

The restricted root system $\bar{R}\subset \mathfrak{X}(T/(T\cap H))$ 
defines restricted Weyl chambers in $\mathfrak{a}_s$, and recall that 
such a Weyl chamber is a fundamental domain for the action of $K$ on $G/H$.

We define $\mathcal{H}$ and $\mathcal{P}$ on $\mathfrak{t}$ by $\mathcal{H}(t)=(t+\sigma(t))/2$ and $\mathcal{P}(t)=(t-\sigma(t))/2$. 
It provides an isomorphism \[(\mathcal{H},\mathcal{P}):\mathfrak{X}(T)\otimes \bbR \rightarrow \mathfrak{X}(T/T_s)\otimes \bbR \oplus \mathfrak{X}(T/(T\cap H))\otimes \bbR .\]

\subsection{Data associated to real line bundles and Kähler forms}

The Picard group of a general spherical variety is described in \cite{Bri89}.
First, the Picard group of horosymmetric Fano manifold $X$ is generated by the $B$-stable prime divisors (which are Cartier). 
These divisors are exactly closures of codimension one $B$-orbits in $X$ and are of two types: the $G$-stable prime divisors, which are also closures of codimension one $G$-orbits, and the closures in $X$ of codimension one $B$-orbits contained in the open $G$-orbit $G/H$, called \emph{colors}. 

The linear relations between these divisors are fully encoded by the \emph{spherical lattice} $\mathcal{M}=\mathfrak{X}(T/(T\cap H))$, which consists of the weights of $B$-eigenvalues in the field of rational functions on $G/H$, and the map $\rho$ which sends a $B$-stable divisor to the element of the dual $\mathcal{M}_{\mathbb{Q}}^*$ induced by restriction of the associated divisorial valuation to the $B$-eigenvalues in $\mathbb{C}(G/H)$.
The relations in the Picard group are exactly the $\sum_{D}\rho(D)(m)=0$ for $m\in \mathcal{M}$, where the sum runs over all $B$-stable prime divisors of $X$.

\subsubsection{Case of linearized line bundles}

Let $\mathcal{L}$ denote a $G$-linearized line bundle on $X$. Then as in \cite{DelHoro} we may associate to it 
\begin{itemize}
    \item a Lie algebra character $\chi_{\mathcal{L}}:\lalg{t}\rightarrow \bbC$ defined by $\exp(t)\cdot \xi=e^{\chi_{\mathcal{L}}(\mathcal{H}(t))}\xi$ for $\xi\in \mathcal{L}_{eH}$, called the \emph{isotropy character},
    \item a \emph{special divisor} $D_{\mathcal{L}}$, defined as the $\bbR$-divisor equal to $1/k$ times the divisor of a $B$-semi-invariant section of some tensor power $\mathcal{L}^k$ whose $B$-weight vanishes on $T_s$.
\end{itemize}
Note that we changed a bit the definitions with respect to \cite{DelHoro}, and consider the isotropy character directly as the Lie algebra character of $T$ induced by the more natural isotropy subgroup character. 
By definition, it is an element of $\mathfrak{X}(T/T_s)\otimes \bbR$.
The special divisor $D_{\mathcal{L}}$ is $B$-stable and thus decomposes as $D_{\mathcal{L}}=\sum_{D}n_DD$ where $D$ runs over the $B$-stable prime divisors of $X$. 
Then we associate to $\mathcal{L}$ its \emph{special polytope}, defined as 
\[ \Delta_{\mathcal{L}}=\{m\in\mathcal{M}\otimes\bbR;\rho(D)(m)+n_D\geq 0, \forall D\} \]
Furthermore, to a $K$-invariant Hermitian metric $q$ on $\mathcal{L}$, we may associate two functions:
\begin{itemize}
    \item the \emph{quasipotential} $\phi_q:G\rightarrow \bbR$ defined by     $\phi(g)=-2\log |g\cdot \xi|_{q}$, for some fixed $\xi\in \mathcal{L}_{eH}$,
    \item and the \emph{toric potential} $u_q:\mathfrak{a}_s\rightarrow \bbR$ defined by $u(a)=\phi(\exp(a))$.
\end{itemize}

\subsubsection{Forgetting the linearization}

Note that for any line bundle $\mathcal{L}$ on $X$, there exists a tensor power $\mathcal{L}^k$ which admits a $G$-linearization, and that two linearizations of the same line bundle differ by a character of $G$ \cite{KKLV89,KKLV89}. 
Modifying the $G$-linearizations of $L$ by a character $\beta$ of $G$ and multiplying the reference element $\xi$ by $s\in\bbC^*$ has the following effects. 

The Lie algebra isotropy character does not change, provided we assume that the center $Z(G)$ of $G$ does not act trivially on $X$, which is easily obtained by quotienting $G$ by the fixator of $X$ and we will assume this from now on. 
Indeed, in this situation, the intersection of $H$ and $Z(G)$ is trivial. 
The special divisor is changed to the linearly equivalent divisor obtained by adding the relation associated to the element of $\mathcal{M}\otimes \bbR$ induced by $\beta$.
The special polytope is translated accordingly.
The quasipotential $\phi$ transforms to the function $\phi-2\log|s\beta(g)|$ and $u$ changes accordingly.

As a result, the isotropy character is a well defined data associated to a line bundle, while the special divisor and special polytope are well defined modulo transformation by a character of $G$. 
Similarly, the quasipotential and toric potential are well defined up to addition of a constant and a function associated to a character of $G$, and actually depend only on the curvature of $q$, if one allows for \emph{real} characters of $G$, that is, elements of $\mathfrak{X}(T/(T\cap [G,G]))\otimes \bbR$. 

\subsubsection{Case of arbitrary real $(1,1)$-classes}

Since $X$ is spherical, any real $(1,1)$-class $\Omega$ on $X$ is the class of a 
real line bundle, that is $\Omega=tc_1(L_1)/k_1+(1-t)c_1(L_2)/k_2$ for some 
real number $t$, integers $k_1$, $k_2$ and line bundles $L_1$ and $L_2$. 
It is then straightforward to extend the definitions: 
the isotropy character of $\Omega$ is $\chi_{\Omega}=t\chi_{L_1}/k_1+(1-t)\chi_{L_2}/k_2$, 
the special divisor is $D_{\Omega}=tD_{L_1}/k_1+(1-t)D_{L_2}/k_2$ and the special 
polytope is defined from the special divisor as before. Note again that the isotropy 
character is well defined, while the last two are defined 
only modulo the action of a \emph{real} character of $G$ (an element of 
$\mathfrak{X}(T/(T\cap [G,G]))\otimes \bbR$). 
Similarly, given a $K$-invariant $(1,1)$-form $\omega\in \Omega$, we can 
associate to it a quasipotential and a toric potential, well defined up to 
a constant and the action of a real character of $G$.

In particular, the above constructions apply to Kähler classes and $K$-invariant 
Kähler forms. The Kähler assumption on the other hand will impose additional 
conditions such as the convexity of the toric potential $u$. 

It is possible to fix the real character by requiring that the derivative of $u$ at the 
origin is zero in directions coming from the action of the center of $G$. 
We will fix a normalization of the toric potentials when relevant. 

\subsection{Special polytope, toric polytope and moment polytope}
\label{sec:corresp_pol}

Let us recall the relations between different polytopes associated 
to a real line bundle on a horosymmetric manifold (see \cite{DelHoro} for details).  

Assume $\Omega$ is the class of a real line bundle. 
The polytope $\Delta_{\Omega}^+:=\chi_{\Omega}+\Delta_{\Omega}$ is called the moment polytope of $\Omega$. 
It coincides with 
Kirwan's moment polytope for $\Omega$, or equivalently Brion's moment 
polytope for $\Omega$ if $\Omega$ is the class of an ample $\bbQ$-line bundle. 
Conversely, the special divisor is the image of the moment polytope under the natural projection $\mathfrak{X}(T)\otimes \bbR\rightarrow 
\mathfrak{X}(T_s)\otimes\bbR$ and the identification of the latter 
with $\mathfrak{X}(T/(T\cap H))\otimes \bbR$.

Let $\omega$ be a $K$-invariant Kähler form in $\Omega$. 
Let us now explain the relation between the toric potential $u$ of $\omega$ 
and the special polytope $\Delta_{\Omega}$.  
There is another natural polytope associated to it, the \emph{toric polytope} $\Delta^{\tor}_{\Omega}$, 
defined as the $\bar{W}$-invariant polytope which is the convex hull of the 
images by the restricted Weyl group $\bar{W}$ of the special polytope.
Again, the toric polytope is well defined only up to translation by a real 
character of $G$. 
Now the relation between the toric potential of a Kähler form in $\Omega$ and 
the toric polytope is 
\[ \{d_au;a\in\mathfrak{a}_s\}=\mathrm{Int}(-2\Delta^{\tor}) \]
up to translation by an element of $\mathfrak{X}(T/(T\cap [G,G]))\otimes \bbR$.

Finally recall the equivalence between the conditions $\Delta^+_{\Omega}\cap \mathfrak{X}(T/[L,L])\otimes\bbR \neq 0$ and $\Delta^+_{\Omega}-\chi_{\Omega}=\Delta_{\Omega}=\Delta^{\tor}_{\Omega}\cap \bar{C}^+$.

\subsection{The connected group of equivariant automorphisms, and Hamiltonian functions}
\label{sec:HamiltHoro}

Consider the action by multiplication on the right of the normalizer 
$N_G(H)$ of $H$ on $G/H$, that is $n\cdot gH=gn^{-1}H$ for $g\in G$, $n\in N_G(H)$. 
This action obviously factorizes through $N_G(H)/H$ and the action of the neutral 
component of this group extends to $X$. It may further be identified with the action 
of the neutral component of the group of $G$-equivariant automorphisms of $X$ 
(see \cite{DelKSSV}).
From \cite{BP87}, we know that $N_G(H)/H$ is diagonalizable and more precisely, 
its Lie algebra is identified with $\mathfrak{Y}((T\cap N_G(H))/(T\cap H))\otimes \bbC$.

The manifold $X$ is thus actually equipped with an action of the group 
$\tilde{G}=G\times N_G(H)/H$, and we let $\tilde{K}$ denote the maximal compact subgroup 
of $\tilde{G}$ obtained as a product of $K$ and the maximal compact subgroup 
of $N_G(H)/H$. 

Let $\omega$ denote a $\tilde{K}$-invariant Kähler form on $X$. 
Let $\chi$, $\phi$ and $u$ denote the associated data. 
Let $V$ denote a holomorphic vector field on $X$ commuting with the action of $G$, and generating a purely non-compact subgroup of automorphisms. 
Let $f_{V,\omega}$ denote the function on $X$ defined, up to normalization, by $L_V\omega=i\partial\bar{\partial}f_{V,\omega}$. 
We naturally identify $V$ with an element of $\mathfrak{Y}((T_s \cap N_G(H)))\otimes \bbR \subset \mathfrak{a}_s$.

\begin{prop}
\label{prop:HamiltonianHoro}
The function $f_{V,\omega}$ is up to a normalizing additive constant equal 
to the $\tilde{K}$-invariant function defined by 
$f_{V,\omega}(\exp(a)\cdot x)= -d_a u(V)$
for $a\in \mathfrak{a}_s$.
\end{prop}

\begin{proof}
Let $\pi:G\rightarrow G/H$ denote the quotient map. Note that the quasipotential 
$\phi$ is an $i\partial\bar{\partial}$-potential for $\pi^*\omega$ on $G$.
Furthermore, $\pi$ is equivariant for the action of $T_s\cap N_G(H)$ on the 
right on $G/H$ and $G$, so that 
\[ i\partial\bar{\partial}L_V\phi= L_V\pi^*\omega 
=\pi^*L_V\omega= i\partial\bar{\partial}(f_{V,\omega}\circ\pi) \]
It is straightforward, using the definition of the quasipotential, to compute 
that $(L_V\phi)(ke^ah)=-d_au(V)$ for any $a\in \lalg{a}_s$, $k\in K$, $h\in H$. 
As a consequence, the difference between $-d_au(V)$ and $f_{V,\omega}\circ\pi$ 
is a bounded $K\times H$-invariant pluriharmonic function on $G$. 
In other words, this difference is a constant.
\end{proof}

\subsection{How to determine special polytopes and isotropy character}

We will in the following and for our main result only be interested in classes of real line bundles satisfying the following assumption. 

\begin{assumption} 
\label{assumption_isotropy}
We assume that the isotropy character of the class is trivial on $[\mathfrak{l},\mathfrak{l}]\cap \mathfrak{h}$. 
\end{assumption} 

By linearity, it is enough to determine the special polytopes on some real generators of the subspace of $(1,1)$-classes satisfying the above assumption. 
The $G$-stable prime divisors are the special divisors of the line bundles they determine, since those admit unique $G$-invariant sections (up to a character of $G$). 
For the same reasons, their isotropy characters are trivial. 

From the description of colors of horosymmetric manifolds in \cite{DelHoro}, it is then enough to consider in addition those colors which are the inverse image of the single color $D$ by some $f:G/H\rightarrow G/P_{\alpha}$ where $P_{\alpha}$ is the maximal proper parabolic subgroup of $G$ containing $B^-$ with $\alpha$ as unique root of $P^u$.
In this case, there is an obvious $B$-semi-invariant section whose divisor is $f^{-1}(D)$: the pull-back of the unique $B$-semi-invariant section of $\mathcal{O}_{G/P_{\alpha}}(D)$. 
The latter has $B$-weight equal to 
the fundamental weight $\varpi$ of $\alpha$. 
Hence the special divisor associated to $\mathcal{O}_X(\overline{f^{-1}(D)}$ is 
$f^{-1}(D)+\mathrm{div}(\varpi\circ\mathcal{P})$, 
and the isotropy character coincides with $\varpi\circ\mathcal{H}$.

\subsection{The anticanonical line bundle}
\label{sec:anticanonical}

For the anticanonical line bundle, which is of utmost importance here, the special divisor was obtained from Brion's anticanonical divisor \cite{Bri97} 
in \cite{DelHoro}: it is the divisor with coefficient zero for colors coming from the symmetric fiber, and $m_D-\sum_{\alpha\in \Phi_{Q^u}\cup \Phi_s^+} \alpha\circ\mathcal{P}(\rho(D))$ for the others, where $m_D=1$ if the divisor is $G$-stable and $m_D=\alpha^{\vee}(\sum_{\beta\in \Phi_s^+\cup \Phi_{Q^u}}\beta)$ if $D$ comes from the simple root $\alpha\in \Phi_{Q^u}$. 
Its isotropy character is $\sum_{\alpha\in \Phi_{Q^u}}\alpha\circ\mathcal{H}$.

Unlike general line bundles, the anticanonical line bundle on $X$ admits a canonical $G$ or $\tilde{G}$-linearization, induced by the natural linearization of the action of the full automorphism group on the tangent bundle.  
We denote by $\Delta_{ac}$, $\Delta^{\tor}_{ac}$, etc the canonical choice of representatives (in the classes of equivalence under translation by a character) of data associated to the anticanonical line bundle, induced by this particular linearization. 
Similarly, for a metric on the anticanonical line bundle there is a canonical choice of toric potential up to an additive constant. 
More precisely, given a metric $h$ and a non-zero element $\xi$ in the fiber of $K_X^{-1}$ over $x$, the toric potential of the metric is, up to an additive constant, the function defined by 
\[ u(a)=-\log \lvert \exp(a)\cdot \xi\rvert^2_h \]
where the group action is the (anti)canonical one. 

Recall also the correspondence between Hermitian metrics on the anticanonical line bundle and volume forms. 
Let $s$ denote a local trivialization of the anticanonical line bundle, and $s^{-1}$ the induced local trivialization of the canonical line bundle. 
Then $i^{n^2}s^{-1}\wedge \overline{s^{-1}}$ defines a local reference volume form. 
A Hermitian metric $h$ on $K_X^{-1}$ is associated with the volume form $\Omega_h$ defined locally by 
\[ \Omega_h = \lvert s_1 \rvert_h^2 i^{n^2}s^{-1}\wedge \overline{s^{-1}} \]
and the converse correspondence is obvious by the same formula. 
In the case of the volume form $\omega^n$ induced by a Kähler form $\omega$, it should be further noticed that the curvature form of the metric associated with $\omega^n$ is none other than the Ricci form $\Ric(\omega)$ of $\omega$. 

\subsection{Monge-Ampère operator}

We describe here the Monge-Ampère operator on Kähler forms provided the Kähler classes satisfy Assumption~\ref{assumption_isotropy}. 

The differential $d_au $ of a function defined on $\lalg{a}_s$ is identified with an element of $\mathfrak{X}(T)\otimes\bbR$ via the projection $\mathcal{P}$.
Given a fixed choice of basis for $\lalg{a}_s$, the real Monge-Ampère operator $u\mapsto \det(d_a^2u)$ is well defined as a real valued function for any $a\in \lalg{a}_s$. 
Let us denote by $I$ the function defined on $\lalg{a}_s$ by 
\[ I(a)=\prod_{\alpha\in \Phi_{Q^u}}e^{-2\alpha(a)}
\prod_{\beta\in \Phi_s^+} \sinh(-2\beta(a)).\] 
Let $P_{DH}$ denote the Duistermaat-Heckman polynomial defined on $\mathfrak{X}(T)\otimes \bbR$ by 
\[ P_{DH}(p)=\prod_{\alpha\in \Phi_{Q^u}\cup \Phi_s^+}\kappa\left(\alpha,p\right), \] 
where $\kappa$ denote the Killing form. 
Let $\lalg{a}_s^-$ denote the negative restricted Weyl chamber, defined as the set of all $a\in \lalg{a}_s$ such that $\beta(a)\leq 0$ for all $\beta\in \Phi_s^+$.

\begin{thm}{\cite{DelHoro}}
\label{thm:MAHoro}
There is a choice of non zero element $\xi$ in the fiber over $x$ of the anticanonical bundle of $G/H$ such that for any $a\in \mathrm{Int}(\lalg{a}_s^-)$, for any $K$-invariant Kähler form satisfying Assumption~\ref{assumption_isotropy}, we have 
\[ 
I(a) (\omega^n)_{e^a\cdot x} = 
\det(d^2_au) P_{DH}(2\chi-d_au)
i^{n^2}((e^a\cdot\xi)^{-1}\wedge \overline{(e^a\cdot\xi)^{-1}})
\]
\end{thm}

Note that while \cite{DelHoro} proves the result only for rational classes, 
it is immediate to extend the result to arbitrary Kähler forms using our 
definition of isotropy character and toric potential. 

\subsection{Translating the equations}

We have now gathered enough from \cite{DelHoro} to translate the system of equations in terms of the toric potentials. We place ourselves in the setting of Section~\ref{sec:intro_horo}.

In particular, each vector field $V_i$ commutes with the action of $G$, and as such is identified with an element of $\mathfrak{Y}(T_s\cap N_G(H))\otimes \bbR$. 
We denote by $\chi_i$ the isotropy character of $\theta_i$. 
Let $w$ denote a fixed toric potential for $\gamma$.
Let $u_i$ denote a toric potential for $\omega_i$, and assume, without loss of generality, that the function $w+\sum_i u_i$ is the canonical (up to constant) choice of toric potential for a metric on the anticanonical line bundle. 
We further assume that the toric potentials are defined consistently within a fixed class, i.e. that the differences 
$ u_i^{\heartsuit}-u_i^{\diamondsuit} $
are bounded functions. 
In particular, the function on $\mathfrak{a}_s$ corresponding to the difference $\phi_i^t$ is $u_i^t-u_i^{\rf}$.

Let $C_i$ be the constants defined by 
$ C_i=\int_{-(2\Delta_i^{\tor}\cap \bar{C}^+)} p(V_i)P_{DH}(2\chi_i-p)dp.$
Let $G_i$ denote the function defined by 
$ G_i(p)= Ce^{h_i(-p(V_i)+C_i)}P_{DH}(2\chi_i-p)$,
where $C$ is the constant which ensures $\int_{-(2\Delta_{\omega_i}^{\tor}\cap \bar{C}^+)}G_i(p)dp=1$. The constant $C$ is independent of $i$ by assumption on the functions $h_i$.
We set  $J(a):=e^{-w(a)}I(a)$.
Finally, fix the normalization in the toric potentials of the form $\sum_i u_i^{\heartsuit}$ by the condition 
\[ \int_{\mathfrak{a}_s^-}e^{-\sum_i u_i^{\heartsuit}(a)}J(a)da = 1. \]

Using the results on horosymmetric manifolds recalled earlier in the section and these notations, we obtain the precise relation between the system of complex Monge-Ampère equations \eqref{eq:ComplexContPath} and the system of real Monge-Ampère equations \eqref{eq:RealContPath}. 

\begin{thm}
\label{thm:translation_horo}
Assume the functions $\phi_i^t$ are solutions to Equation~\eqref{eq:ComplexContPath}, then the functions $u_i^t$ are solutions, on $\mathrm{Int}(\lalg{a}_s^-)$, to the system of equations 
\[  G_i(d_au_i^t)\det(d^2_au_i^t)
= e^{-\sum_{l=1}^{k}(tu_l^t+(1-t)u_l^{\rf})(a)} 
J(a).\]
\end{thm}

\begin{proof}
Let $\xi$ be the element of $(K_X^{-1})_x$ produced by Theorem~\ref{thm:MAHoro}, then by the same result we have for any $i$, 
\[ 
I(a)((\omega_i^t)^n)_{e^a\cdot x}= 
\det(d^2_au_i^t) P_{DH}(2\chi_i-d_au_i^t)i^{n^2}((e^a\cdot\xi)^{-1}\wedge \overline{(e^a\cdot\xi)^{-1}}).
\]
By Proposition~\ref{prop:HamiltonianHoro}, and the choice of normalization of Hamiltonian
we have 
$f_{V_i,\omega_i^t}(e^a\cdot x)=-d_au_i^t(V_i)+C_i$ 
for the constant $C_i$ independent of $t$ defined before the statement of the theorem. 

We now turn to the right hand side of Equation~\ref{eq:ComplexContPath}.
By definition of $\theta_0$, by the assumption on $w+\sum_i u_i^{\heartsuit}$, and by the correspondences recalled in Section~\ref{sec:anticanonical},  
\[ (\theta_0^n)_{e^a\cdot x} = 
e^{-w(a)+\sum_j u_j^{\rf}(a)}i^{n^2}((e^a\cdot\xi)^{-1}\wedge \overline{(e^a\cdot\xi)^{-1}}). \]
By definition of $\phi_i^t$, we then have
\[ (e^{-t\sum \phi_j}\theta_0^n)_{e^a\cdot x} = 
e^{-w(a)}e^{t\sum_j u_j^t(a)+(1-t)\sum_j u_j^{\rf}(a)}i^{n^2}((e^a\cdot\xi)^{-1}\wedge \overline{(e^a\cdot\xi)^{-1}}) \]
and it concludes the proof.
\end{proof}

We are now able to carry out the proof of Theorem~\ref{thm:cpld_can_horo}.

\begin{proof}[Proof of Theorem~\ref{thm:cpld_can_horo}]
It follows from the successive application of Theorem~\ref{thm:ReductionToC0},  of Theorem~\ref{thm:translation_horo}, and of Theorem~\ref{thm:C0Estimates}, then a few manipulations on the conditions. 
We thus only need to check that the data provided by Theorem~\ref{thm:translation_horo} verify the assumptions of Theorem~\ref{thm:C0Estimates}.

For this we identify the remaining data: we are working on the vector space $\lalg{a}=\lalg{a}_s$, the cone is $\lalg{a}^+=\lalg{a}_s^-$, and the polytopes $\Delta_i$ are the  $\Delta_i=-2(\Delta^{\tor}_{\theta_i}\cap \bar{C}^+)$. 
The functions $g_k$ is a product of a polynomial with the exponential of an affine function hence satisfy Assumption~\ref{Jassumption}~(\ref{gass_integrable}). 
For the function $J$, note that 
\[ j_{\infty}=v_{-2(\Delta_{\gamma}^{\tor}\cap \bar{C}^+)}+2\sum_{\alpha\in \Phi_{Q^u}\cup \Phi_s^+}\alpha. \]
It is thus continuous and finite on $\lalg{a}^+$, $j-j_{\infty}$ is bounded from below on the whole cone $\lalg{a}^+$ since $\log\sinh(x)\leq x$. Finally, Assumption~\ref{Jassumption}~(\ref{Jass_proper}) is satisfied since we always have, by the description of the anticanonical divisor,  $2\sum_{\alpha\in \Phi_{Q^u}\cup \Phi_s^+}\alpha\circ\mathcal{P} \in \mathrm{Int}(2\Delta_{ac}^{\tor})=\mathrm{Int}(2\Delta_{\gamma}^{\tor}+2\sum_i\Delta_{\theta_i}^{\tor})$.
\end{proof}

\section{Examples}
\label{sec:examples}

We will illustrate in this section some of the consequences of our main result, but this also serves us as a pretext to present the natural horosymmetric structure on several Fano threefolds. 
This is even easier now that the connected automorphism groups of Fano threefolds have been completely determined in \cite{CPS18}. 

Recall that there are five homogeneous Fano threefolds: the quadric $Q$, products of projective spaces $\bbP^3$, $\bbP^2\times \bbP^1$, $(\bbP^1)^3$ and the full flag threefold, usually denoted by $W$.
It turns out from examination of the remaining threefolds with infinite connected automorphism groups that the non-toric, possibly spherical threefold, necessarily have a connected \emph{reductive} automorphism group isogenous to $\SL_2\times \bbC^*$. 
Any horosymmetric threefold under the action of $\SL_2\times \bbC^*$ is either symmetric or rank two horospherical. 
It is not hard to check that any horospherical \emph{threefold} has to be toric.

On the other hand, toric manifolds in general are horospherical under the action of a maximal connected \emph{reductive} automorphism group, and considering this more precise horospherical structure reduces the rank of the action unless the latter group is a torus. 
For toric Fano threefolds which are not homogeneous, the possible connected reductive automorphism groups are in the list: $\SL_3\times \bbC^*$, $\SL_2^2\times \bbC^*$, $\SL_2\times (\bbC^*)^2$. 

In this section we will determine the symmetric Fano threefolds, as well as the best horospherical structures on toric Fano threefolds. 
For the latter, we will focus on the Fano threefolds equipped with a rank one horospherical structure, and only mention that all the others either appear in the list of rank $2$ smooth and Fano embeddings of $\SL_2/U\times \bbC^*$ in Pasquier's thesis \cite[Chapitre~7]{Pas06},
or are products of $\bbP^1$ with a toric Del Pezzo surface. 
We will then study some rank one Fano $\SL_2\times \SL_3$-horospherical \emph{fourfolds} including the toric example of Fano manifold with no Kähler-Einstein metric but a pair of coupled Kähler-Einstein metrics obtained by the second author in \cite{Hul}. 

In this section, we will often identify Fano threefolds by their identifier as used for example in \cite{CPS18}. 
We will use for this the notation $F^3_{I}$ for the threefold with identifier $I$. 
For example, $F^3_{2.33}$ denotes the blowup of the projective space $\bbP^3$ along a line. 
Note that the number before the period in the identifier is the Picard rank of the threefold. 

In this section, the group $G$ will be a product of special linear groups and of one-dimensional tori. 
We fix as maximal torus $T$ the subgroup of diagonal matrices, and as Borel subgroup $B$ the subgroup of upper-diagonal matrices. 
The positive root system $\Phi^+\subset \mathfrak{X}(T)$ will be the one associated to  these choices. 

A horospherical homogeneous space $G/H$ is fully determined by the data of the parabolic $P=N_G(H)$ and of the spherical lattice $\mathcal{M}\subset \mathfrak{X}(P)$. 
Furthermore, Pasquier obtained very explicit criterions to determine if a horospherical embedding is Fano and smooth. 
These allow to determine the horospherical structures on smooth Fano threefolds easily.
For symmetric threefolds, we will use the classification of Fano symmetric varieties with low rank obtained by Ruzzi \cite{Ruz12}.

Finally, we will work in this section rather with moment polytope than toric polytopes, and refer to Section~\ref{sec:corresp_pol} for the correspondence with the toric polytope which allows to apply the results as stated in Section~\ref{sec:intro_horo}.

\subsection{Warm up: Hirzebruch surfaces and their blow-down}

Before considering threefolds, let us first consider the non-homogeneous horospherical surfaces. 
The only possibility for the group action is $G=\SL_2$. 
It is of rank one, type $A_1$, we use the notations $\Phi^+=\{\alpha\}$, $\mathfrak{X}(T)=\bbZ \frac{\alpha}{2}$ and have (up to a constant)  $P_{DH}(x\frac{\alpha}{2})=x$. 
The normalizer $P$ of $H$ has to be the opposite Borel subgroup (the lower triangular matrices), but there remains a choice of the generator $\mu$ of the lattice $\mathcal{M}\subset \mathfrak{X}(P)=\mathfrak{X}(T)$. 
We use the notation $\mu=k\alpha/2$ for some positive integer $k$. 

There are always two choices of $G$-equivariant embeddings of $G/H$ in this situation, for each $k$: one toroidal embedding $\mathbb{F}_k$ and one colored embedding $\bbP(1,1,k)$. 
The toroidal ones are none other than the Hirzebruch surfaces. 
In particular, $\mathbb{F}_1$ is the blow up of $\bbP^2$ at one point, and for higher $k$, $\mathbb{F}_k$ is not Fano. 
On the other hand the colored are, as the notation suggests, weighted projective planes. They are Fano, but singular as soon as $k\geq 2$. 
The anticanonical moment polytope of $\bbP(1,1,k)$ is 
$\Delta=[0,(1+k/2)\alpha]\subset \mathfrak{X}(T)\otimes \bbR.$ 

Set $M:=1+k/2$, $V(a,b)=\int_0^Mt(at+b) dt$, and $X(a,b)=\int_0^M t^2(at+b)dt$ for $a,b\in \bbR$. 
The necessary condition for existence of Mabuchi metrics is satisfied 
on these weighted projective planes if for the unique pair $(a,b)$ such that 
$X(a,b)=V(a,b)=V(0,1)$,  
the affine function $t\mapsto at+b$ is positive on $[0,M]$. 

A straightforward computation using $V(a,b)=aM^3/3+bM^2/2$, $X(a,b)=aM^4/4+bM^3/3$
shows that $b=3(3M-4)/M$, $a=6(3-2M)/M^2$. 
We recover that $a=0$ in the case $k=1$ which corresponds to $\bbP^2$, 
and in this situation $b=1>0$ so that unsurprisingly there exists a Kähler-Einstein metric.
In the cases $k>1$, one has $0<-b/a<M$, 
so that the corresponding affine function is not positive on the polytope, that is there cannot be any reasonable singular Mabuchi metric on $\bbP(1,1,k)$.  

\subsection{Fano rank one $\SL_2\times \SL_2$-horospherical threefolds}

The group $\SL_2\times \SL_2$ is of rank two, type $A_1\times A_1$. We have $\Phi^+=\{\alpha_1,\alpha_2\}$ and  $\mathfrak{X}(T)=\bbZ \frac{\alpha_1}{2}+\bbZ \frac{\alpha_2}{2}$.
The only generalized flag manifold for the action of $\SL_2\times \SL_2$ is $\bbP^1\times \bbP^1$. 
The parabolic $P$ is the Borel subgroup consisting of pairs of lower triangular matrices (opposite to $B$), thus $\mathfrak{X}(P)=\mathfrak{X}(T)$. 
For any horospherical subgroup $H$ with normalizer equal to $P$, we have have the same Duistermaat-Heckman, and its expression is up to a multiplicative constant $P_{DH}(x\frac{\alpha_1}{2}+y \frac{\alpha_2}{2})=xy$.
From the work of Pasquier \cite{Pas06,Pas08}, 
we deduce the list of choices of $\mathcal{M}$ that allow smooth and Fano embeddings of $G/H$ and list these. 
Up to obvious symmetries, we have the following possibilities. 

The first possibility is the product case, if  $\mathcal{M}=\bbZ \alpha_1/2$. 
In this case, $G/H$ is the product $\bbC^2\setminus \{0\}\times \bbP^1$ and the smooth and Fano embeddings are the products $F^3_{3.28}=\mathbb{F}_1\times \bbP^1$ and $F^3_{2.34}=\bbP^2\times \bbP^1$. 
Their moment polytopes are $\{t\alpha_1/2+\alpha_2; t\in [t_-,3]\}$  where $t_-=0$ for $F^3_{2.34}$ and $t_-=1$ for $F^3_{3.28}$.

The second possibility is $\mathcal{M}=\bbZ \frac{\alpha_1+\alpha_2}{2}$. 
The corresponding horospherical homogeneous space admits a unique smooth and Fano embedding: the Fano threefold $F^3_{3.31}$ which is the $\SL_2\times \SL_2$-equivariant $\bbP^1$-bundle $\bbP(\mathcal{O}\oplus\mathcal{O}(1,1))$ on $\bbP^1\times \bbP^1$.
The moment polytope is $\{t(\alpha_1+\alpha_2)/2;t\in [1,3]\}$.

The last possibility for $\mathcal{M}$ is $\bbZ\frac{\alpha_1-\alpha_2}{2}$. There are then three (up to symmetry) smooth and Fano embeddings: the fully colored one, which is the projective space $\bbP^3$, the single colored $F^3_{2.33}$ and the toroidal $F^3_{3.25}$. 
The geometrical description of the action is straighforward: 
consider $\bbC^4$ as the sum $\bbC^2\oplus \bbC^2$ and the corresponding componentwise action 
of $\SL_2\times \SL_2$. 
The induced action on $\bbP^3$ is the rank one horospherical structure, with two colored closed orbits which are two disjoint lines: the natural inclusions of the projectivization of both $\bbC^2$ summands. 
Then $F^3_{2.33}$ is the blow-up of $\bbP^3$ along one of these 
lines, while $F^3_{3.25}$ is the blow up along both lines. 
The moment polytopes are $\{(4-t)\alpha_1/2+t\alpha_2/2; t\in [t_-,t_+]\}$ where $(t_-,t_+)$ is $(0,4)$ for $\bbP^3$, $(0,3)$ for $F^3_{2.33}$, and $(1,3)$ for $F^3_{3.25}$.

Let us now see how the conditions for existence of Kähler-Einstein metrics or Mabuchi metrics on these manifolds reduce to simple computations in one variable in this setting. Note that the results here are not new (see \cite{NSY17}).
The manifold $F^3_{3.25}$ is known to be Kähler-Einstein. 
This follows from the obvious symmetry of its moment polytope and of $P_{DH}$ under the exchange of $\alpha_1$ and $\alpha_2$.

For Mabuchi metrics on $F^3_{3.31}$, we consider an arbitrary affine function on the polytope, which we may write $t\mapsto at+b$ under the parametrization used before. 
Then the barycenter involved in the search for Mabuchi metric is 
\[ \mathbf{bar}(a,b):=\frac{\int_1^3 (at+b)t^3dt}{\int_1^3 (at+b)t^2dt}\frac{\alpha_1+\alpha_2}{2}=\frac{363a+150b}{300a+130b}(\alpha_1+\alpha_2) \]
It follows first that it is different from $\alpha_1+\alpha_2$ if $a=0$ \emph{i.e.} there are no Kähler-Einstein metrics. 
More precisely, the barycenter is $\alpha_1+\alpha_2$ if and only if $63a+20b=0$, and under this condition the affine function $t\mapsto at+b$ is positive on $[1,3]$, so there are Mabuchi metrics on $F^3_{3.31}$.

Similarly, for $F^3_{2.33}$ define a barycenter depending on a choice of affine function $t\mapsto at+b$ by 
\[ \mathbf{bar}(a,b):=2\alpha_1+\frac{\int_0^3 (at+b)t^2(4-t)dt}{\int_0^3 (at+b)t(4-t)dt}\frac{\alpha_2-\alpha_1}{2} = 2\alpha_1+\frac{72a+35b}{10(7a+4b)}(\alpha_2-\alpha_1) \]
The barycenter coincides with $\alpha_1+\alpha_2$ if and only if $2a-5b=0$, and under this condition the affine function $t\mapsto at+b$ is positive on $[0,3]$, so there are Mabuchi metrics on $F^3_{2.33}$ as well.

\subsection{Fano rank one $\SL_3$-horospherical threefolds}

The group $\SL_3$ is of rank two, type $A_2$. 
We have $\Phi^+=\{\alpha_1,\alpha_2,\alpha_1+\alpha_2\}$ and $\mathfrak{X}(T)=\bbZ \frac{\alpha_1+2\alpha_2}{3}+\bbZ \frac{2\alpha_1+\alpha_2}{3}.$

The only flag manifold of dimension $2$ under $\SL_3$ is the projective space $\bbP^2$. 
We may thus assume that $P$ is conjugate to the stabilizer of a plane in $\bbC^3$. 
We assume that $-\alpha_1$ is a root of $P^u$. 
The characters of $P$ are those characters in the weight lattice of $\SL_3$ which are orthogonal to $\alpha_2$ with respect to the Killing form. 
In other words, this is the one dimensional $\bbZ(2\alpha_1+\alpha_2)/3$. 
The Duistermaat-Heckman polynomial in this section will be $P_{DH}(x\frac{\alpha_1+2\alpha_2}{3}+y \frac{2\alpha_1+\alpha_2}{3})=xy(x+y) $.
Using again the work of Pasquier, we deduce the possible horospherical homogeneous spaces with smooth and Fano embeddings, as well as these embeddings. 

The first case is $\mathcal{M}=\bbZ (2\alpha_1+\alpha_2)/3$. 
There are two smooth and Fano embeddings of the corresponding horospherical homogeneous space. 
The colored one is isomorphic to $\bbP^3$, with moment polytope $\{t(2\alpha_1+\alpha_2)/3; t\in[0,4]\}$.
The toroidal one is the $\bbP^1$-bundle $F^3_{2.35}=\bbP(\mathcal{O}\oplus\mathcal{O}(1))$ over $\bbP^2$, with moment polytope $\{t(2\alpha_1+\alpha_2)/3; t\in[2,4]\}$.

The second case is $\mathcal{M}=\bbZ (4\alpha_1+2\alpha_2)/3$. 
In this case, the toroidal embedding $F^3_{2.36}=\bbP(\mathcal{O}\oplus\mathcal{O}(2))$ over $\bbP^2$, with moment polytope $\{t(2\alpha_1+\alpha_2)/3; t\in[1,5]\}$ is the only smooth and Fano embedding. 
It is easy to check that $F^3_{2.36}$ does not admit any Mabuchi metric. We consider for this the barycenter, depending on the variables $a$ and $b$ defining the affine function $t\mapsto at+b$, defined by 
\[ \mathbf{bar}(a,b):=\frac{\int_1^5(at+b)t^3dt}{\int_1^5(at+b)t^2dt}\frac{2\alpha_1+\alpha_2}{3}=\frac{781a+195b}{585a+155b}(2\alpha_1+\alpha_2) \]
The above barycenter is equal to $2\alpha_1+\alpha_2$ if and only if $196a+40b=0$, but under this condition, the affine function $t\mapsto at+b$ is not positive on the segment $[1,5]$. 

However if we consider the variant of Mabuchi metrics given by considering powers of affine functions, we get the existence of a canonical metric of this type on $F^3_{2.36}$ as soon as we consider the squares of affine functions. 
To see this, consider 
\[ \mathbf{bar}^{(2)}(a,b):=\frac{\int_1^5(at+b)^2t^3dt}{\int_1^5(at+b)^2t^2dt}\frac{2\alpha_1+\alpha_2}{3}=\frac{3255a^2+1562ab+195b^2}{2343a^2+1170ab+155b^2}(2\alpha_1+\alpha_2) \]
The above barycenter is equal to $2\alpha_1+\alpha_2$ if and only if the second order equation $912a^2+398ab+40b^2=0$ is satisfied. 
Choosing (without loss of generality) $b=912$, there are two solutions $a_{\pm}=-199\pm\sqrt{3121}$. 
The affine function $a_-t+b$ changes sign on $[1,5]$, but not $a_+t+b$ which remains positive. 
 
\subsection{Symmetric Fano threefolds under the action of $\SL_2\times \bbC^*$}

The group $G=\SL_2\times \bbC^*$ is of rank two, type $A_1$. 
We have $\Phi^+=\{\alpha\}$ and $\mathfrak{X}(T)=\bbZ \frac{\alpha}{2}\oplus \bbZ f$ (where $f$ is a generator of $\mathfrak{X}(T/[G,G])$).

There is only one possible rank two involution on $G$ (up to conjugation): 
$\sigma(g,s)=((g^t)^{-1},s^{-1})$ for $(g,s)\in \SL_2\times \bbC^*$.
Up to quotienting the $\bbC^*$ factor so that it acts effectively, 
there are three possible symmetric subspaces: 
$H_A:=\SO_2\times \{1\}$, $H_B:=N(\SO_2)\times \{1\}$, and the intermediate 
case $H_C:=\{(g,g\sigma(g)^{-1}/I_2),g\in N(\SO_2)\}$.
For the last subgroup, $I_2$ denotes the identity matrix in $\SL_2$, 
and the quotient is well defined since for any $g\in N(\SO_2)$, 
$g\sigma(g)^{-1}= \pm I_2 \in Z(\SL_2)$.
In any case, the Duistermaat-Heckman polynomial is given (as always up to a multiplicative constant) by  $P_{DH}(x\frac{\alpha}{2}+yf)=x$.

From Ruzzi \cite{Ruz12} we know the smooth and Fano $G$-equivariant embeddings of these symmetric spaces. 
This provides 12 Fano threefolds. We draw in Figures~\ref{fig:RA}, \ref{fig:RB} and \ref{fig:RC} the moment polytopes, and include the data of the spherical lattice $\mathcal{M}$ as the dotted grid, while the dashed line represents the single wall of the positive restricted Weyl chamber in $\mathcal{M}\otimes \bbR$.
We denote by RA1, RA2, RA3 the Fano embeddings of $G/H_A$, by RB1, RB2, RB3 the Fano embeddings of $G/H_B$, and by RC1 to RC6 the Fano embeddings of $G/H_C$ as in the figures.
\begin{figure}
\begin{tikzpicture}
\draw (-.5,-.5) node{RA1}; 
\draw [dotted] (-1,-2) grid[xstep=1,ystep=1] (3,2);

\draw (0,0) node{$\bullet$};
\draw [very thick, ->] (0,0) -- (1,0) node[above right]{$\alpha$};

\draw [dashed] (0,-2) -- (0,2);

\draw [very thick] (0,-1) -- (2,-1) -- (2,1) -- (0,1) -- cycle;
\end{tikzpicture}
\begin{tikzpicture}
\draw (-.5,-.5) node{RA2};
\draw [dotted] (-1,-2) grid[xstep=1,ystep=1] (3,2);

\draw (0,0) node{$\bullet$};
\draw [very thick, ->] (0,0) -- (1,0) node[above right]{$\alpha$};

\draw [dashed] (0,-2) -- (0,2);

\draw [very thick] (0,-1) -- (3,-1) -- (1,1) -- (0,1) -- cycle;
\end{tikzpicture}
\begin{tikzpicture}
\draw (-.5,-.5) node{RA3};
\draw [dotted] (-1,-2) grid[xstep=1,ystep=1] (3,2);

\draw (0,0) node{$\bullet$};
\draw [very thick, ->] (0,0) -- (1,0) node[above right]{$\alpha$};

\draw [dashed] (0,-2) -- (0,2);

\draw [very thick] (0,-1) -- (2,-1) -- (2,0) -- (1,1) -- (0,1) -- cycle;
\end{tikzpicture}
\caption{Moment polytopes for smooth and Fano embeddings of $G/H_A$ \label{fig:RA}}
\end{figure}
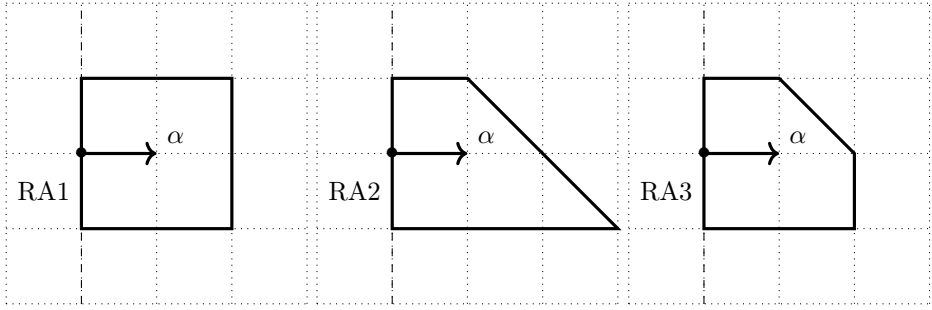

\begin{figure}
\begin{tikzpicture}
\newcommand*\size{1}
\draw (-.5,-.5) node{RB1};
\draw [dotted] (-1*\size,-2*\size) grid[xstep=\size,ystep=\size] (3*\size,2*\size);

\draw (0,0) node{$\bullet$};
\draw [very thick, ->] (0,0) -- (\size/2,0) node[above right]{$\alpha$};

\draw [dashed] (0,-2*\size) -- (0,2*\size);

\draw [very thick] (0,-\size) -- (3*\size/2,-\size) -- (3*\size/2,\size) -- (0,\size) -- cycle;
\end{tikzpicture}
\begin{tikzpicture}
\newcommand*\size{1}
\draw (-.5,-.5) node{RB2};
\draw [dotted] (-1*\size,-2*\size) grid[xstep=\size,ystep=\size] (3*\size,2*\size);

\draw (0,0) node{$\bullet$};
\draw [very thick, ->] (0,0) -- (\size/2,0) node[above right]{$\alpha$};

\draw [dashed] (0,-2*\size) -- (0,2*\size);

\draw [very thick] (0,-\size) -- (5*\size/2,-\size) -- (\size/2,\size) -- (0,\size) -- cycle;
\end{tikzpicture}
\begin{tikzpicture}
\newcommand*\size{1}
\draw (-.5,-.5) node{RB3};
\draw [dotted] (-1*\size,-2*\size) grid[xstep=\size,ystep=\size] (3*\size,2*\size);

\draw (0,0) node{$\bullet$};
\draw [very thick, ->] (0,0) -- (\size/2,0) node[above right]{$\alpha$};

\draw [dashed] (0,-2*\size) -- (0,2*\size);

\draw [very thick] (0,-\size) -- (3*\size/2,-\size) -- (3*\size/2,0) -- (\size/2,\size) -- (0,\size) -- cycle;
\end{tikzpicture}
\caption{Moment polytopes for smooth and Fano embeddings of $G/H_B$ \label{fig:RB}}
\end{figure}
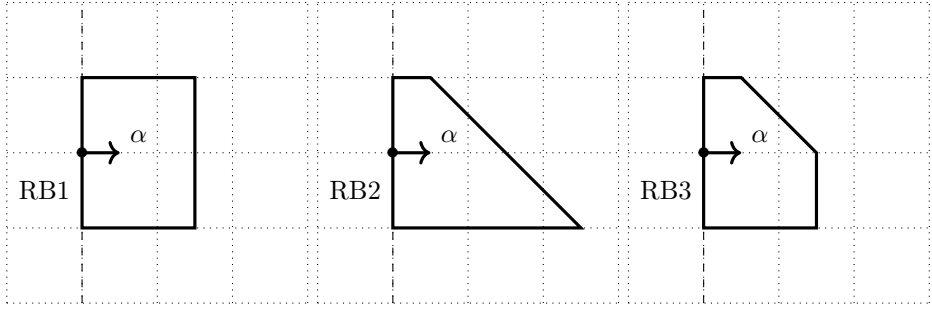

\begin{figure}
\begin{tikzpicture}
\newcommand*\size{1}
\draw (-.5,.5) node{RC1};
\draw [dotted] (-2*\size,-3*\size) grid[xstep=\size,ystep=\size] (2*\size,2*\size);

\draw (0,0) node{$\bullet$};
\draw [very thick, ->] (0,0) -- (\size/2,-\size/2) node[above right]{$\alpha$};

\draw [dashed] (-2*\size,-2*\size) -- (2*\size,2*\size);

\draw [very thick] (-3*\size/2,-3*\size/2) -- (3*\size/2,-3*\size/2) -- (3*\size/2,3*\size/2) -- cycle;
\end{tikzpicture}
\begin{tikzpicture}
\newcommand*\size{1}
\draw (-.5,.5) node{RC2};
\draw [dotted] (-2*\size,-3*\size) grid[xstep=\size,ystep=\size] (2*\size,2*\size);

\draw (0,0) node{$\bullet$};
\draw [very thick, ->] (0,0) -- (\size/2,-\size/2) node[above right]{$\alpha$};

\draw [dashed] (-2*\size,-2*\size) -- (2*\size,2*\size);

\draw [very thick] (-\size/2,-\size/2) -- (3*\size/2,-5*\size/2) -- (3*\size/2,3*\size/2) -- cycle;
\end{tikzpicture}
\begin{tikzpicture}
\newcommand*\size{1}
\draw (-.5,.5) node{RC3};
\draw [dotted] (-2*\size,-3*\size) grid[xstep=\size,ystep=\size] (2*\size,2*\size);

\draw (0,0) node{$\bullet$};
\draw [very thick, ->] (0,0) -- (\size/2,-\size/2) node[above right]{$\alpha$};

\draw [dashed] (-2*\size,-2*\size) -- (2*\size,2*\size);

\draw [very thick] (-\size/2,-\size/2) -- (3*\size/2,-5*\size/2) -- (3*\size/2,-\size/2) -- (\size/2,\size/2) -- cycle;
\end{tikzpicture}

\begin{tikzpicture}
\newcommand*\size{1}
\draw (-.5,.5) node{RC4};
\draw [dotted] (-2*\size,-2*\size) grid[xstep=\size,ystep=\size] (2*\size,2*\size);

\draw (0,0) node{$\bullet$};
\draw [very thick, ->] (0,0) -- (\size/2,-\size/2) node[above right]{$\alpha$};

\draw [dashed] (-2*\size,-2*\size) -- (2*\size,2*\size);

\draw [very thick] (-3*\size/2,-3*\size/2) -- (\size/2,-3*\size/2) -- (3*\size/2,-\size/2) -- (3*\size/2,3*\size/2) -- cycle;
\end{tikzpicture}
\begin{tikzpicture}
\newcommand*\size{1}
\draw (-.5,.5) node{RC5};
\draw [dotted] (-2*\size,-2*\size) grid[xstep=\size,ystep=\size] (2*\size,2*\size);

\draw (0,0) node{$\bullet$};
\draw [very thick, ->] (0,0) -- (\size/2,-\size/2) node[above right]{$\alpha$};

\draw [dashed] (-2*\size,-2*\size) -- (2*\size,2*\size);

\draw [very thick] (-\size/2,-\size/2) -- (\size/2,-3*\size/2) -- (3*\size/2,-3*\size/2) -- (3*\size/2,3*\size/2) -- cycle;
\end{tikzpicture}
\begin{tikzpicture}
\newcommand*\size{1}
\draw (-.5,.5) node{RC6};
\draw [dotted] (-2*\size,-2*\size) grid[xstep=\size,ystep=\size] (2*\size,2*\size);

\draw (0,0) node{$\bullet$};
\draw [very thick, ->] (0,0) -- (\size/2,-\size/2) node[above right]{$\alpha$};

\draw [dashed] (-2*\size,-2*\size) -- (2*\size,2*\size);

\draw [very thick] (-\size/2,-\size/2) -- (\size/2,-3*\size/2) -- (3*\size/2,-3*\size/2) -- (3*\size/2,-\size/2) -- (\size/2,\size/2) -- cycle;
\end{tikzpicture}
\caption{Moment polytopes for smooth and Fano embeddings of $G/H_C$ \label{fig:RC}}
\end{figure}
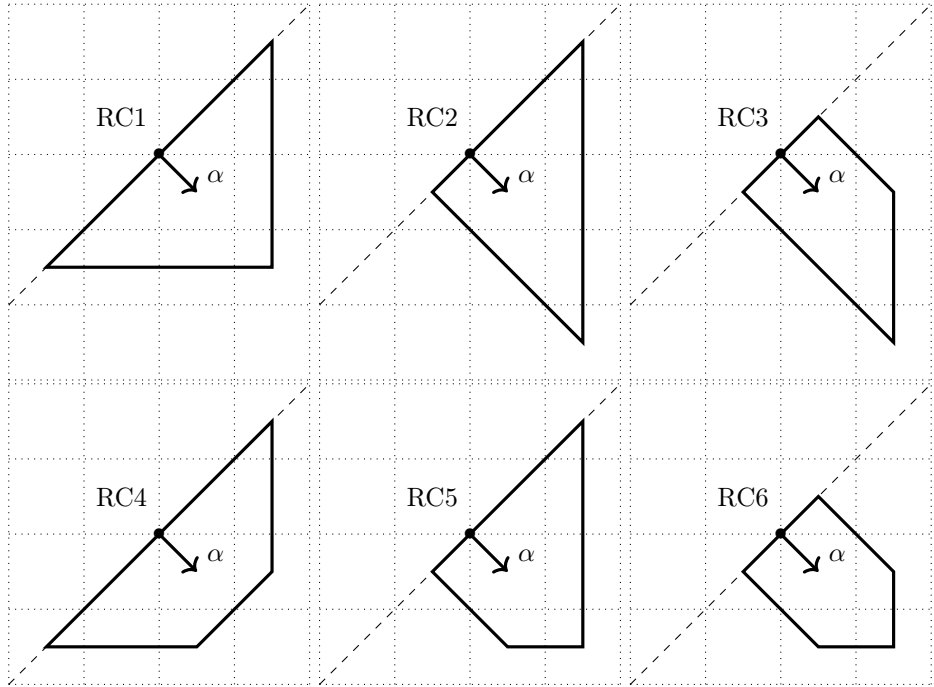

Given their moment polytopes, it is straightforward to compute (using \cite{Bri89}) their Picard rank and anticanonical degree. 
It is often enough to identify the Fano threefolds. 
Namely, this allows to immediately identify RA2, RA3, RB2, RB3, RC1, RC2, and RC3. 
On the other hand, it is clear from the combinatorial data that RA1 and RB1 are the products of $\bbP^1$ with the wonderful compactification of their two dimensional symmetric factor (respectively $\SL_2/\SO_2\subset \bbP^1\times \bbP^1$ and $\SL_2/N_{\SL_2}(\SO_2)\subset \bbP^2$). 
For RC4, it follows from \cite{CPS18} that $F^3_{2.29}$ is the only threefold with appropriate Picard rank, anticanonical degree, and that admits an \emph{almost effective} action of $\SL_2\times \bbC^*$. 
For the last two, let us provide a geometrical description which allows to identify them. 

Consider the standard quadratic form $z_1^2+z_2^2+z_3^2+z_4^2+z_5^2$ on $\bbC^5$. 
Write $\bbC^5=\bbC^3\times \bbC^2$ and note that the above quadratic form is the sum of the standard quadratic forms on both factors. 
Consider the induced action of $P(O_3\times O_2)\simeq \GL_2$ on the quadric in $\bbP^4$. 
It admits five orbits: the closed orbit given by inclusion of the $1$-dimensional quadric from the first factor, the two fixed points given by the two components of the $0$-dimensional quadric from the second factor, the codimension $1$ orbit consisting of elements which are the image of products of non-zero elements of both affine quadrics, and finally the open orbit, isomorphic to $G/H_C$. 
This action coincides with the structure of $G/H_C$-embedding of RC1. 
From the caracterization of equivariant morphisms between $G/H_C$-embeddings \cite[Theorem~4.1]{Kno91}, 
we deduce that RC5 is the blowup of RC1 at one of the fixed point, while RC6 is the blowup of RC1 at both fixed points. Note also that RC4 is the blowup of RC1 along the $1$-dimensional quadric. 

We summarize the correspondence in the following table:\\
\begin{tabular}{cccccccccccc}
\toprule
RA1&RA2&RA3&RB1&RB2&RB3&RC1&RC2&RC3&RC4&RC5&RC6 \\
\midrule
$F^3_{3.27}$ & $F^3_{3.31}$ & $F^3_{4.8}$ & $F^3_{2.34}$ & $F^3_{2.36}$ & $F^3_{3.22}$ & $Q^3$ & $\bbP^3$ & $F^3_{2.35}$ & $F^3_{2.29}$ & $F^3_{2.30}$ & $F^3_{3.19}$ \\
\bottomrule
\end{tabular}

Note in particular that RA3, RB3, RC4, RC5, RC6 are neither homogeneous nor toric, and that $\SL_2\times \bbC^*$ is up to isogeny a maximal connected reductive automorphism group of each. 

To illustrate our results on these examples, let us consider the existence of Mabuchi metrics on those threefolds. 
It is known that RC4 and RC6 admit Kähler-Einstein metrics. For the other three, they are the lowest dimensional examples where our results give new existence or non-existence for Mabuchi metrics. 

For RC5, we write points in the moment polytope as $xf+y\alpha$ where $x$ and $y$ are parameters, $f$ is orthogonal to $\alpha$ and the vertices of the moment polytope are given by their coordinates $(x,y)$ as the elements of the set $\{(-3,0),(0,3),(1,2),(1,0) \}$. 
Searching for a Mabuchi metric amounts to searching for an affine function $x\mapsto ax+b$ (note the dependence only in $x$), positive on $[-3,1]$, such that 
\[ \left(\int_{-3}^0\int_0^{3+x}+\int_0^1\int_0^{3-x}\right) \left((xf+(y-1)\alpha)(ax+b)y\right) dy dx=x_bf+y_b\alpha \]
satisfies $x_b=0$ and $y_b>0$. 
The condition on $x_b$ imposes $49a-20b=0$, and under this condition, the affine function vanishes at $-49/20\in [-3,1]$. 
As a consequence, RC5 does not admit any Mabuchi metric. 

Consider now RA3, we again fix coordinates $(x,y)$ with respect to a basis $(f,\alpha)$ where $f$ is orthogonal to $\alpha$ and such that the coordinates of the vertices of the moment polytope are $(-1,0)$, $(-1,2)$, $(0,2)$, $(1,1)$ and $(1,0)$. 
We search for an affine function $x\mapsto ax+b$, positive on $[-1,1]$, such that 
\[ \left(\int_{-1}^0\int_0^{2}+\int_0^1\int_0^{2-x}\right)\left((xf+(y-1)\alpha)(ax+b)y\right) dy dx=x_bf+y_b\alpha \]
satisfies $x_b=0$ and $y_b>0$. 
The condition $x_b=0$ implies $112a-65b=0$ and the other conditions are then satisfied, so RA3 admits a Mabuchi metric. 

Similar verifications show that RB3 also admits a Mabuchi metric. 

\subsection{Rank one $\SL_2\times \SL_3$-horospherical Fano fourfolds and  Futaki's example}

The group $\SL_2\times \SL_3$ is of rank three and of type $A_1\times A_2$. 
We have $\Phi^+=\{\alpha_1,\alpha_2,\alpha_3,\alpha_2+\alpha_3\}$ and $\mathfrak{X}(T)=\bbZ\varpi_1\oplus \bbZ \varpi_3+\bbZ \varpi_2$ where $\varpi_1=\frac{\alpha_1}{2}$, $\varpi_2=\frac{2\alpha_2+\alpha_3}{3}$ and $\varpi_3=\frac{\alpha_2+2\alpha_3}{3}$.

We are interested in rank one, four dimensional horospherical homogeneous spaces that do not split as products. 
As a consequence, we assume that $P$ is the parabolic with $\Phi_{P^u}=\{-\alpha_1,-\alpha_2\}$, and that $\mathcal{M}=\bbZ e \subset \mathfrak{X}(P)=\bbZ\varpi_1+\bbZ\varpi_2$ where $e=k_1\varpi_1+k_2\varpi_2$ for some non-zero relative integers $k_1$ and $k_2$. 
From Pasquier's work, there are smooth and Fano embeddings  exactly for $(k_1,k_2)\in\{\pm(1,1),\pm(-1,1),\pm(1,2),\pm(-1,2)\}$, including, for each, the toroidal $\bbP^1$-bundle $\bbP(\mathcal{O}\oplus\mathcal{O}(k_1,k_2))$ over $\bbP^1\times \bbP^2$.

Let us concentrate first on the case $(k_1,k_2)=(-1,1)$, whose toroidal embedding corresponds to Hultgren's example \cite{Hul}. 
The toroidal embedding $X=\bbP(\mathcal{O}\oplus\mathcal{O}(1,-1))$ has Picard rank three. 
By the description of the Picard group for horospherical varieties, and the criterion for ampleness, 
we can choose parameters $(s_1,s_2,s_3)$ for $0\leq s_1\leq s_2\leq s_3$ parametrizing semi-positive real line bundles $L(s_1,s_2,s_3)$ (with associated class of $B$-stable divisor represented by $(s_3-s_1)D_+ +(s_2-s_3)D_- + s_3D_{\alpha_1}$). 
The moment polytope of $L(s_1,s_2,s_3)$ is then 
\[ \Delta(s_1,s_2,s_3)=\{s_3\varpi_2+t(\varpi_1-\varpi_2); s_1\leq t \leq s_2 \}. \] 
In particular, the anticanonical moment polytope is $\Delta(1,3,5)$. 

\begin{figure}
\begin{tikzpicture}
\draw [dotted] (-3,-1) grid[xstep=1,ystep=1] (6,6);
\draw (0,0) node{$\bullet$};
\draw (0,0) node[below left]{$0$};
\draw[very thick] (1,4) -- (3,2) node[above right]{$\Delta(s_1,s_2,s_3)$};
\draw[dashed] (0,5) -- (5,0);
\draw[dashed] (1,0) -- (1,4);
\draw[dashed] (3,0) -- (3,2);
\draw (0,5) node{$\bullet$};
\draw (0,5) node[above left]{$s_3\varpi_2$};
\draw (5,0) node{$\bullet$};
\draw (5,0) node[below]{$s_3\varpi_1$};
\draw (1,0) node{$\bullet$};
\draw (1,0) node[below]{$s_1\varpi_1$};
\draw (3,0) node{$\bullet$};
\draw (3,0) node[below]{$s_2\varpi_1$};
\draw (5,5) node[above right]{$\mathfrak{X}(P)$};
\draw (1,-1) -- (-3,3) node[above right]{$\mathcal{M}\otimes \bbR$};
\end{tikzpicture}
\caption{Moment polytopes for $\bbP_{\bbP^1\times \bbP^2}(\mathcal{O}\oplus \mathcal{O}(-1,1))$}
\end{figure}
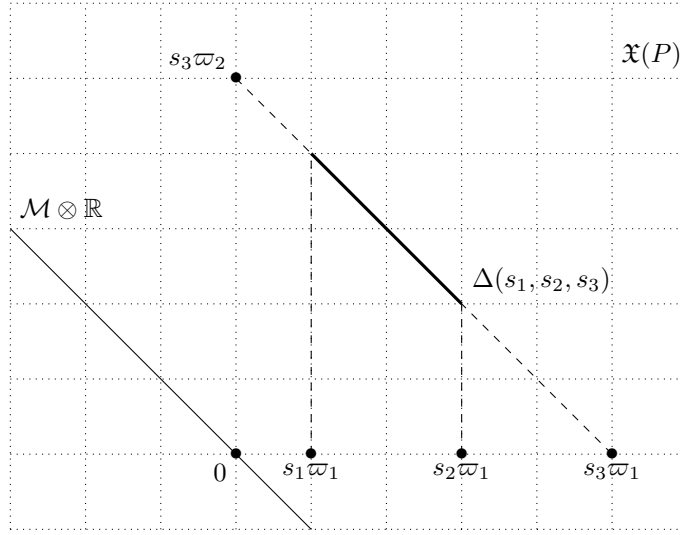

The Duistermaat-Heckamn polynomial on $\mathfrak{X}(P)\otimes\bbR$ is, up to a multiplicative constant, given by 
$P_{DH}(x_1\varpi_1+x_2\varpi_2)=x_1x_2^2$.
The barycenter of $\Delta(s_1,s_2,s_3)$ with respect to the Duistermaat-Heckman polynomial is 
\[ \mathbf{bar}(s_1,s_2,s_3)=s_3\varpi_2+\frac{2}{5}\frac{6(s_2^5-s_1^5)-15s_3(s_2^4-s_1^4)+10s_3^2(s_2^3-s_1^3)}{3(s_2^4-s_1^4)-8s_3(s_2^3-s_1^3)+6s_3^2(s_2^2-s_1^2)}(\varpi_1-\varpi_2) \]
In particular, for the anticanonical line bundle, $\mathbf{bar}(1,3,5)=5\varpi_2+\frac{244}{125}(\varpi_1-\varpi_2)\neq 2\varpi_1+3\varpi_2= 2\rho_H$, so we recover that $X$ is not Kähler-Einstein. 

For a decomposition of the anticanonical line bundle in a sum of two ample real line bundles $L(1,3,5)=L(s_1,s_2,s_3)+L(1-s_1,3-s_2,5-s_3)$ (valid if $0<s_1<1$, $s_1<s_2<2+s_1$ and $s_2<s_3<2+s_2$)
we see that the equation $\mathbf{bar}(s_1,s_2,s_3)+\mathbf{bar}(1-s_1,3-s_2,5-s_3)=2\rho_H$ boils down to a quartic equation in $s_3$ provided $s_1$ and $s_2$ are fixed, so that we can precisely determine if there is a solution as well as its exact value.  For the obvious choice $s_1=1/2$ and $s_2=3/2$ we do not find a solution, so we arbitrarily consider $s_1=1/4$ and $s_2=3/2$, to obtain the quartic equation 
\[ 30720s_3^4-184000s_3^3+386272s_3^2-348246s_3+115587=0.  \]
It admits two real roots (with complicated expression), one of which is approximately $2,6831$ and satisfies the Kähler condition for the decomposition. 
As a consequence, there exist pairs of coupled Kähler-Einstein metrics on this Fano fourfold. 

To conclude on this homogeneous space $G/H$, let us describe all four smooth and Fano embeddings whose existence is ensured by \cite{Pas08}. 
The elementary description is as follows. 
Consider $\bbC^5$ as the product $\bbC^2\times \bbC^3$, and the induced action of $\SL_2\times \SL_3$ on $\bbP^4$. 
There are three orbits under this action: the line $\bbP(\bbC^2\times\{0\})$, the plane $\bbP(\{0\}\times\bbC^3)$, and an open orbit isomorphic to $G/H$. 
The smooth and Fano embeddings of $G/H$ are the projective space $\bbP^4$ and its blowups along the line, the plane, and both. 
This fourth possibility coincides with the toroidal embedding considered above. 

Let us now turn to the case $(k_1,k_2)=(-1,2)$ which corresponds to the embedding $X=\bbP(\mathcal{O}\oplus\mathcal{O}(-1,2))$. 
In this case, we may again describe the semipositive real line bundles $L_{s_1,s_2,s_3}$ using three real parameters $0\leq s_1\leq s_2 \leq s_3$. 
The anticanonical line bundle is $L(1,3,7)$.
The moment polytope is \[ \Delta(s_1,s_2,s_3)=\{s_3\varpi_2+t(\varpi_1-2\varpi_2); s_1\leq t \leq s_2\}. \] 
Let $\mathbf{bar}(s_1,s_2,s_3)$ denote again the barycenter of $\Delta(s_1,s_2,s_3)$ with respect to the Duistermaat-Heckman polynomial. 

Let us consider the decompositions of the anticanonical of the form $L(1,3,7)=L(1/2,3/2,z)+L(1/2,3/2,7-z)$, which is a decomposition into Kähler classes for $3/2<z<11/2$. 
The equation $\mathbf{bar}(1/2,3/2,z)+\mathbf{bar}(1/2,3/2,7-z)=2\varpi_1+3\varpi_2$ translates to the quartic equation 
\[ 10(z^2-7z)^2+261(z^2-7z)+1631=0 \]
which admits two real solutions $z_{\pm}$ 
\[ \frac{3}{2}<z_-=\frac{35-\sqrt{5(\sqrt{2881}-16)}}{10}<z_+=\frac{35-\sqrt{5(\sqrt{2881}-16)}}{10}<\frac{11}{2}. \]
We thus obtain two examples of decompositions with coupled Kähler-Einstein metrics. 
Note that $z=7/2$ is not a solution of the above equation, so there are no Kähler-Einstein metrics on $X$.

\bibliographystyle{alpha}
\bibliography{HoroCMab}
\end{document}